\documentclass[12pt]{article}
\usepackage[utf8]{inputenc}

\usepackage[left=2.5cm,right=2.5cm, top=2.5cm,bottom=2.5cm,bindingoffset=0cm]{geometry}

\usepackage{tikz, tikz-cd}
\usetikzlibrary{calc,intersections,through,backgrounds}
\usetikzlibrary{angles,quotes,decorations.markings}

\usepackage{amsmath, amssymb}
\usepackage{amsthm, mathtools, hyperref}

\usepackage[export]{adjustbox}

\mathtoolsset{showonlyrefs}

\usepackage[sort&compress]{natbib}

\usepackage{float}

\newtheorem{lemma}{Lemma}[section]
\newtheorem{proposition}[lemma]{Proposition}
\newtheorem{theorem}[lemma]{Theorem}
\newtheorem{corollary}[lemma]{Corollary}
\newtheorem{problem}[lemma]{Problem}

\theoremstyle{definition}
\newtheorem{remark}[lemma]{Remark}
\newtheorem{definition}[lemma]{Definition}

\usepackage{subcaption}

\usetikzlibrary{arrows.meta}

\usetikzlibrary{positioning}
\usetikzlibrary{decorations.text}
\usetikzlibrary{decorations.pathmorphing}
\usetikzlibrary{decorations.markings}

\usepackage{tkz-euclide}

\usepackage{caption}
\usepackage[nottoc]{tocbibind}

\newcommand{\per}{\mathbf T}
\newcommand{\hp}{{T_1}}
\newcommand{\vp}{{T_2}}

\newcommand{\C}{\mathbb {C}}
\newcommand{\CP}{\mathbb {CP}}
\newcommand{\Z}{\mathbb Z}

\renewcommand{\S}{{{\mathcal S}}}

\newcommand{\PGL}{\mathrm {PGL}}
\newcommand{\GL}{\mathrm {GL}}
\newcommand{\Aff}{\mathrm {Aff}}

\newcommand{\tr}{\mathrm{tr}\,}
\newcommand{\eps}{\varepsilon}
\newcommand{\proj}{ {P}}
\newcommand{\id}{\mathbf{1}}

\newcommand{\FG}{G}
\newcommand{\Mat}{\mathrm{Mat}}
\graphicspath{{figs//}}

\title{{Integrability of Cauchy problems for discrete conformal maps and circle patterns}}
\author{Maxim Arnold\thanks{Department of Mathematical Sciences, University of Texas at Dallas, e-mail: \tt{Maxim.Arnold@utdallas.edu}}~
and Anton Izosimov\thanks{School of Mathematics and Statistics, University of Glasgow, e-mail: \tt{Anton.Izosimov@glasgow.ac.uk}}}
\date{}

\begin{document}

\maketitle
\vspace{-7mm}
\abstract{
A map from a square lattice to the Riemann sphere is called discrete conformal if the image of every elementary square is a harmonic quadrilateral. We prove that the initial value problem for discrete conformal maps with quasi-periodic boundary conditions is Liouville integrable. We also show that the image of the embedding of Schramm's orthogonal square grid circle patterns into the space of discrete conformal maps is the real part of a symplectic leaf. As a consequence, we obtain the integrability of the corresponding Cauchy problem for circle patterns.
}

\tableofcontents

\section{Introduction}

Discretization of classical complex analysis is one of the central problems in discrete differential geometry. While there exist different approaches to constructing such discretizations \cite{bobenko2005linear}, in this paper we are concerned with the so-called \textit{nonlinear theory} of discrete conformal maps.

A map \(f \colon \mathbb{Z}^2 \to \mathbb{CP}^1\) is called \textit{discrete conformal} if the image of every elementary lattice square is a harmonic quadrilateral. In other words, \(f\) is discrete conformal if it satisfies the \textit{cross-ratio equation}
\begin{equation}\label{eq:CRE}
[f_{i,j}, f_{i+1,j}, f_{i+1,j+1}, f_{i,j+1}] = -1
\quad \text{for all } i,j \in \mathbb{Z},
\end{equation}
where
\begin{equation}\label{eq:cross-ratio}
[a,b,c,d] := \frac{(a-b)(c-d)}{(b-c)(d-a)}.
\end{equation}
is the cross-ratio (here and below we use index notation for maps defined on \(\mathbb{Z}^d\)). The motivation for this definition is that, in the continuous setting, conformal maps can be characterized as those that map infinitesimal squares to harmonic quadrilaterals.

The cross-ratio equation \eqref{eq:CRE} first appears in~\cite{nijhoff1995discrete} in connection with the discrete KdV equation. In~\cite{bobenko1996discrete}, maps satisfying the cross-ratio equations arise in the context of discrete isothermic surfaces and are called \textit{discrete holomorphic functions}. The name \textit{discrete conformal maps} appears in~\cite{bobenko1999discrete}, where such maps are related to Schramm's \textit{square grid circle patterns} \cite{schramm1997circle}. The latter can themselves be regarded as discretizations of conformal maps and are discussed in detail below.

The topic of the present paper is the \emph{integrability} of the
cross-ratio equation. It is well known that the cross-ratio equation is
\emph{3D consistent} \cite{bobenko2002integrable} and therefore admits a
zero-curvature formulation, which is often taken as a definition of
integrability. Here, however, we are concerned with the integrability of
the \emph{initial value problem} in the Arnold--Liouville sense.

The natural way to pose an initial value problem for the cross-ratio
equation is to prescribe initial data on a \emph{zigzag}
\cite{adler2004cauchy}, that is, a lattice path in \(\mathbb{Z}^2\) whose
projections onto both coordinate axes are bijections.
The initial value problem on a zigzag is well posed: for a full-measure
subset of initial data prescribed on \(\zeta\), there exists a unique
discrete conformal map extending these data to the entire lattice
\(\mathbb Z^2\).

In the present paper we consider quasi-periodic initial data. Let
\[
\per=(\hp,\vp)\in\mathbb Z^2\setminus\{0\},
\qquad
n:=|\hp|+|\vp|.
\]
We assume that the zigzag is \(\per\)-periodic, that is,
\[
\zeta_{i+n}-\zeta_i=\per,
\qquad i\in\mathbb Z.
\]
See Figure~\ref{fig:zig}. Such zigzags exist if and only if both
\(\hp\) and \(\vp\) are nonzero. Since the cross-ratio equation is invariant
under reversing the coordinate directions, we may assume without loss of
generality that
$
\hp,\vp>0,
$
or equivalently, \(\per\in\mathbb Z_{>0}^2\).

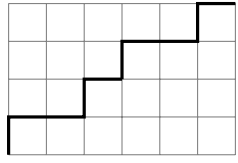
\begin{figure}[htb]
\centering

\begin{tikzpicture}
\draw[step=0.5cm,color=gray] (0,0) grid (3,2);
\draw [very thick] (0,0) -- (0,0.5) -- (1,0.5) -- (1,1) -- (1.5,1);
\begin{scope}[shift={(1.5,1)}]
    \draw [very thick] (0,0) -- (0,0.5) -- (1,0.5) -- (1,1) -- (1.5,1);
\end{scope}
\end{tikzpicture}
\caption{A $(3,2)$-periodic zigzag.}
\label{fig:zig}
\end{figure}

Initial data for a discrete conformal map \(f\) on a zigzag \(\zeta\)
are said to be \emph{quasi-periodic} with period \(\per\) and \emph{monodromy}
\(\Delta\in\mathbb C\) if
\begin{equation}\label{eq:QP}
f_{i+\hp,j+\vp}=f_{i,j}+\Delta
\end{equation}
for every lattice point \((i,j)\in\zeta(\mathbb Z)\). Since the
cross-ratio equation is translation invariant, every discrete conformal
map with quasi-periodic initial data is itself quasi-periodic, in the
sense that \eqref{eq:QP} holds for all \((i,j)\in\mathbb Z^2\).

Quasi-periodic discrete conformal maps with period
\(\per=(\hp,\vp)\) and monodromy \(\Delta\) may be viewed as discrete
conformal maps
\[
\mathbb Z^2/\mathbb Z\per
\longrightarrow
\mathbb C/\mathbb Z\Delta,
\]
from the discrete cylinder to the complex cylinder.


The initial value problem for quasi-periodic discrete conformal maps can
now be formulated as follows. Restricting a quasi-periodic discrete
conformal map with period \(\per=(\hp,\vp) \in \Z^2_{>0}\) to a
\(\per\)-periodic zigzag yields a function
$
p\colon\mathbb Z\to\mathbb C
$
satisfying
\[
p_{i+n}=p_i+\Delta,
\qquad
n:=T_1 + T_2.
\]
We refer to such functions as \emph{quasi-periodic \(n\)-gons}. Equivalently,
they are \emph{twisted polygons} whose monodromy is a translation; see, for
example, \cite{arnold2022cross} for the general definition of twisted
polygons. We denote by
$
\mathcal P_n
$
the space of quasi-periodic \(n\)-gons with arbitrary translational monodromy~\(\Delta\).

Thus, solving the Cauchy problem amounts to recovering a
quasi-periodic discrete conformal map from its restriction
\[
 p = f|_\zeta\in\mathcal P_n
\]
to a \(\per\)-periodic zigzag. This problem admits a unique solution for
a full-measure subset of initial data. Explicitly, the solution is
obtained by iterating the birational \emph{solution map}
\[
\mathcal S\colon\mathcal P_n\dashrightarrow\mathcal P_n,
\]
which assigns to the initial data on a zigzag the corresponding initial
data on the adjacent parallel zigzag; see
Figure~\ref{fig:pzig}. Iterating \(\mathcal S\) recovers the values of
the discrete conformal map at every lattice point. Accordingly,
throughout this paper, integrability of the Cauchy problem means
integrability of the birational map \(\mathcal S\).

The space \(\mathcal P_n\) of quasi-periodic \(n\)-gons carries a
diagonal action of the affine group
\[
\Aff(\mathbb C)
=
\{\,z\mapsto az+b\mid a\in\mathbb C^\times,\ b\in\mathbb C\,\},
\]
given by
\[
(p_i) \mapsto (ap_i + b).
\]
It is therefore natural to consider the
quotient
$
\mathcal P_n/\Aff(\mathbb C),
$
and hence to regard the solution map
\(\mathcal S\) as acting on this quotient.

\begin{figure}[!htb]
\centering
\begin{tikzpicture}
\draw[step=0.5cm,color=gray] (-0.5,-0.5) grid (3.5,2.5);
\draw [very thick](-0.5,-0.5) --  (-0.5,0) -- (0,0) -- (0,0.5) -- (1,0.5) -- (1,1) -- (1.5,1);
\begin{scope}[shift={(1.5,1)}]
    \draw [very thick] (0,0) -- (0,0.5) -- (1,0.5) -- (1,1) -- (1.5,1) --  (1.5,1.5) -- (2,1.5);
\end{scope}
\begin{scope}[shift={(0.5,-0.5)}]
\draw [very thick] (-0.5,0) -- (0,0) -- (0,0.5) -- (1,0.5) -- (1,1) -- (1.5,1);
\begin{scope}[shift={(1.5,1)}]
    \draw [very thick] (0,0) -- (0,0.5) -- (1,0.5) -- (1,1) -- (1.5,1) -- (1.5,1.5);
\end{scope}
\end{scope}
\end{tikzpicture}
\caption{Adjacent parallel zigzags.}
\label{fig:pzig}
\end{figure}

\begin{theorem}\label{thm1}
For every \(\per = (\hp, \vp)\in\mathbb Z_{>0}^2\), the Cauchy problem for
quasi-periodic discrete conformal maps with period \(\per\) is
Arnold--Liouville integrable. Equivalently, letting \(n=\hp + \vp\), the
corresponding initial data space
$
\mathcal P_n/\Aff(\mathbb C)
$
carries a Poisson structure invariant under the solution map
\(\mathcal S\), together with a maximal collection of functionally
independent Poisson-commuting functions preserved by \(\mathcal S\).
\end{theorem}

\begin{remark}
The solution map \(\mathcal S\) depends on the choice of a
\(\per\)-periodic zigzag \(\zeta\). However, if \(\zeta\) and
\(\zeta'\) are two such zigzags, then the corresponding solution maps
are birationally conjugate via the birational identification of the
associated spaces of initial data induced by the cross-ratio equation.
Consequently, Arnold--Liouville integrability is independent of the
choice of zigzag.
\end{remark}

\begin{remark}
Integrability of initial value problems for discrete conformal maps with horizontal
periods, $\per = (\hp, 0)$, 
was studied in
\cite{hetrich2001periodic,arnold2022cross}. In contrast to the
 setting considered here, the corresponding initial value
problem is not well posed: it gives rise to an algebraic correspondence
rather than a birational map.
\end{remark}
\begin{remark}Functions on the space
$
\mathcal P_n/\Aff(\mathbb C)
$
that are invariant under the solution map \(\mathcal S\) have been
known previously. They arise from the zero-curvature representation of
\(\mathcal S\), which in turn is a consequence of the
three-dimensional consistency of the cross-ratio equation
\cite{bobenko2002integrable}. What we prove here is that these invariants are algebraically independent and
Poisson commute with respect to a natural Poisson structure on
\(\mathcal P_n/\Aff(\mathbb C)\). This establishes 
Arnold--Liouville integrability of the Cauchy problem.
\end{remark}





Next, we apply these results to the study of Schramm's orthogonal
square-grid circle patterns~\cite{schramm1997circle}. By definition,
such a pattern consists of two circle packings in the complex plane,
each with the combinatorics of the square grid, that intersect
orthogonally; see Figure~\ref{fig:schramm}.

\begin{figure}[H]
    \centering

        \includegraphics[width=0.4\linewidth]{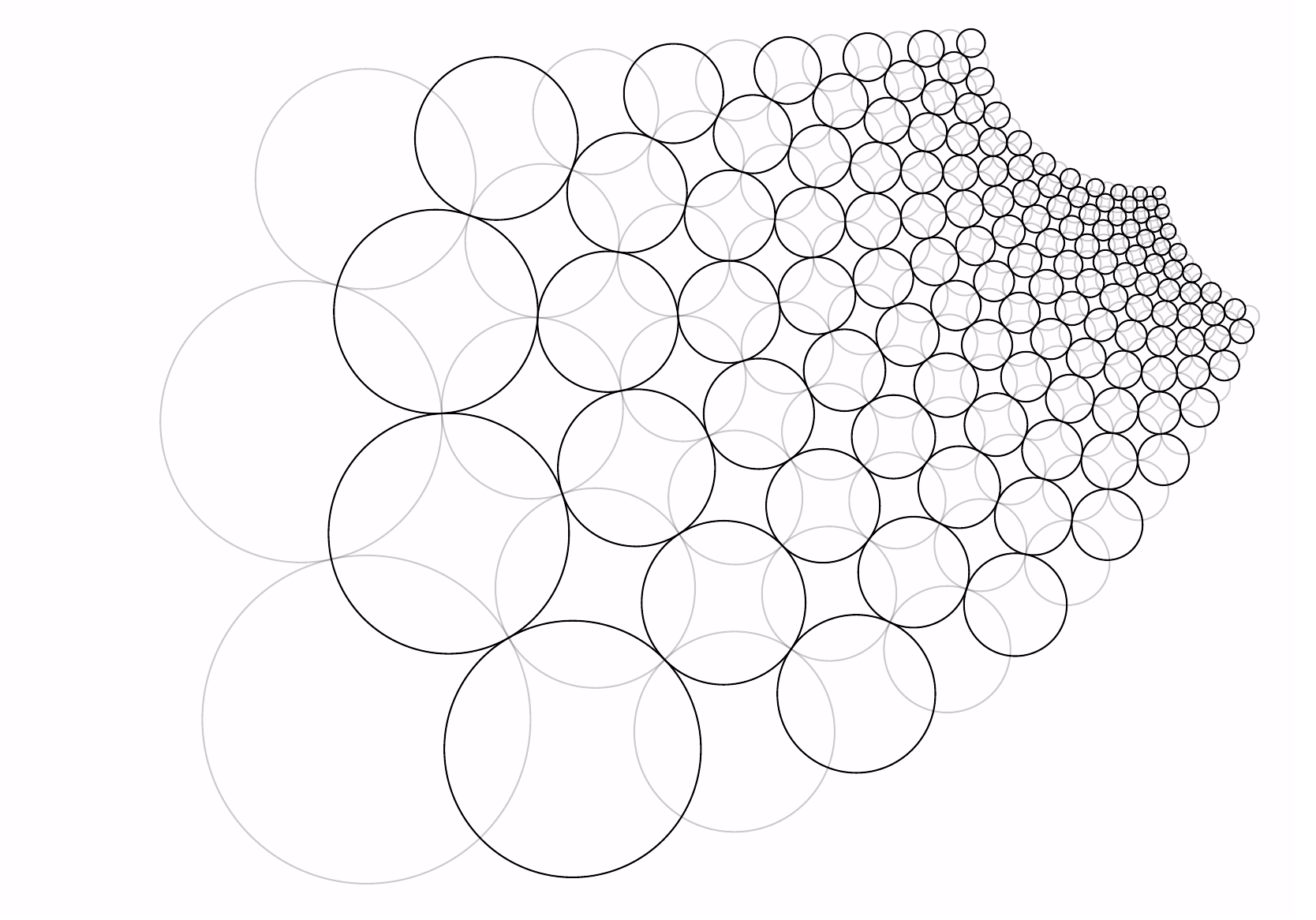}
        \caption{Schramm's circle pattern.}
        \label{fig:schramm}
    \end{figure}



We consider circle patterns that are quasi-periodic, that is, invariant
under translation by a nonzero complex number \(\Delta\); see
Figure~\ref{fig:circle_pattern}. Their initial data consist of a
periodic sequence of mutually orthogonal circles; see
Figure~\ref{fig:Sivp}. Our second main result is the integrability of
the corresponding Cauchy problem.

    \begin{figure}[!htb]
       \centering
    \begin{subfigure}{0.4\linewidth}
        \centering
                 {\includegraphics[width=\linewidth]{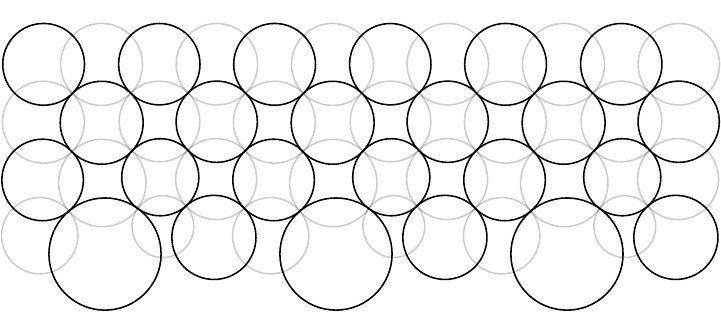}}
        \caption{A quasi-periodic Schramm circle pattern.}
        \label{fig:circle_pattern}
\end{subfigure}
\qquad
 \begin{subfigure}{0.4\linewidth}
\centering

\includegraphics[width=\linewidth,valign=c]{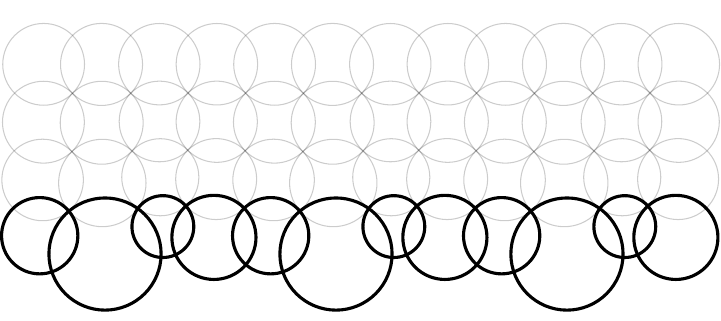}

    \caption{An initial value problem for a Schramm circle pattern. The row of black circles constitutes the initial data. }
    \label{fig:Sivp}
    \end{subfigure}
    \caption{}
\end{figure}

Our approach relies on the correspondence between Schramm's circle
patterns and discrete conformal maps described in
\cite[Section~4]{bobenko1999discrete}. Namely, the circle centers,
together with the intersection points of adjacent circles, form a
discrete conformal map; see Figure~\ref{fig:kites}.

\begin{figure}[!htb]
    \centering
    \includegraphics[width=0.25\linewidth]{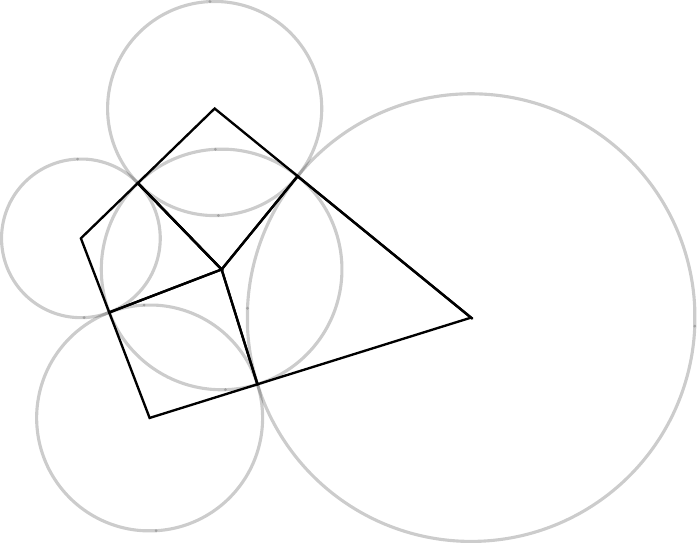}
    \caption{A lattice formed by circle centers and intersection points of adjacent circles in a Schramm circle pattern. Each quadrilateral of the lattice is a right kite and hence harmonic.}
    \label{fig:kites}
\end{figure}

In the quasi-periodic setting, this construction identifies the space of
initial data for Schramm circle patterns with a distinguished subspace of
the space $\mathcal P_{2n}$ of initial data for quasi-periodic discrete
conformal maps, namely the subspace of polygons in which every second side
is equal in length to the preceding side and orthogonal to the following
side; see Figure~\ref{fig:Sivp-ivp}.

\begin{figure}[!htb]
    \centering
   {\includegraphics[width=0.5\linewidth,valign=c]{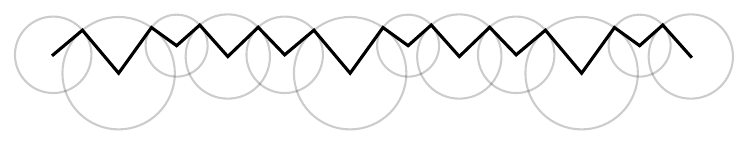}}
   
    \caption{Circle-pattern initial data viewed as a special class of initial data for the cross-ratio equation.}
    \label{fig:Sivp-ivp}

\end{figure}

We prove that the quotient
$
\mathcal K_{2n}/\Aff(\mathbb C),
$
where
$
\mathcal K_{2n}\subset\mathcal P_{2n}
$
denotes this distinguished subspace, is the \emph{real locus of a symplectic
leaf} of
$
\mathcal P_{2n}/\Aff(\mathbb C)
$
with respect to a non-standard real structure. Furthermore, the commuting first integrals from
Theorem~\ref{thm1} are compatible with the corresponding real structure
and therefore restrict to commuting first integrals on
$
\mathcal K_{2n}/\Aff(\mathbb C).
$
This yields the following theorem.

\begin{theorem}\label{thm2}
The Cauchy problem for
quasi-periodic Schramm circle patterns is Arnold--Liouville
integrable. Equivalently, for every \(n\in\mathbb Z_{>0}\), the corresponding initial data space
$
\mathcal K_{2n}/\Aff(\mathbb C)
$
carries a symplectic structure invariant under the solution map
\(\mathcal S\), together with a maximal collection of functionally
independent Poisson-commuting functions preserved by~\(\mathcal S\).
\end{theorem}

\textbf{Acknowledgments.} A.I. is grateful to Max Planck Institute for Mathematics in Bonn for
its hospitality and financial support. A.I. was partially supported by the Simons Foundation
through its Travel Support for Mathematicians program. M.A. was partially supported by the Simons Foundation grant MPS-TSM-00013259.


\section{Discrete conformal maps}
\subsection{Polygons and polygon spaces}

\begin{definition}
A \emph{quasi-periodic \(n\)-gon} is a map
\[
p\colon\mathbb Z\to\mathbb C, \qquad i \mapsto p_i,
\]
such that
\[
p_i\neq p_{i+1}
\qquad\text{for all }i\in\mathbb Z,
\]
and
\[
p_{i+n}=p_i+\Delta
\qquad\text{for all }i\in\mathbb Z,
\]
for some \(\Delta\in\mathbb C\). The complex number \(\Delta\) is called
the \emph{monodromy} of the polygon. We denote the space of quasi-periodic
\(n\)-gons by
$
\mathcal P_n.
$
\end{definition}

The group of translations of \(\mathbb C\) acts freely on
\(\mathcal P_n\). The quotient
\(
\mathcal P_n/\mathbb C
\)
is naturally identified with the torus
$
(\mathbb C^\times)^n,
$
via the edge coordinates
\[
z_i:=p_{i+1}-p_i.
\]
Indeed, the edge vectors are invariant under translations and satisfy
\(
z_{i+n}=z_i.
\)
Conversely, every point of \((\mathbb C^\times)^n\) determines a
quasi-periodic polygon uniquely up to translation.

Similarly, the affine group
\[
\Aff(\mathbb C)
=
\{z\mapsto az+b\mid a\in\mathbb C^\times,\ b\in\mathbb C\}
\]
acts freely on \(\mathcal P_n\). Since translations act trivially on
the edge vectors, the quotient
\(
\mathcal P_n/\Aff(\mathbb C)
\)
is obtained from \((\mathbb C^\times)^n\) by quotienting by the
simultaneous rescaling
\[
(z_1,\dots,z_n)
\longmapsto
(a z_1,\dots,a z_n).
\]
The quotient is therefore naturally identified with the affine variety
\[
\left\{
(y_1,\dots,y_n)\in(\mathbb C^\times)^n
\;\middle|\;
y_1\cdots y_n=1
\right\},
\]
via the coordinates
\[
y_i:=\frac{z_i}{z_{i-1}}.
\]
\subsection{The solution map as a Coxeter element}

In this section, we define the solution map
\(\mathcal S\colon\mathcal P_n\dashrightarrow\mathcal P_n\)
and show that it admits a factorization as a Coxeter element of a
certain Coxeter group acting on quasi-periodic polygons. The generators
of this group are involutive local transformations of polygons, called
\emph{foldings}.

\begin{definition}
Let
\[
\per=(\hp,\vp)\in\mathbb Z_{>0}^2,
\qquad
n=\hp+\vp.
\]
A \emph{\(\per\)-periodic zigzag} is a map
\[
\zeta\colon\mathbb Z\to\mathbb Z^2, \qquad i \mapsto \zeta_i,
\]
such that
\[
\zeta_{i+1}-\zeta_i\in\{(1,0),(0,1)\},
\qquad
\zeta_{i+n}-\zeta_i=\per,
\]
for all \(i\in\mathbb Z\).
\end{definition}

Given $\per=(\hp,\vp)$ and a \(\per\)-periodic zigzag
\(
\zeta\colon\mathbb Z\to\mathbb Z^2,
\)
restriction to \(\zeta\) identifies the space of quasi-periodic
discrete conformal maps of period \(\per\) with a full-measure subset of
the space \(\mathcal P_n\) of quasi-periodic \(n\)-gons, where
\(
n=\hp+\vp.
\)
This identification is a map
\[
r_\zeta\colon
\{\text{quasi-periodic discrete conformal maps of period }\per\}
\longrightarrow
\mathcal P_n.
\]
Let
\[
\zeta'=\zeta+(1,-1)
\]
be the adjacent parallel zigzag. Restriction to \(\zeta'\) similarly
defines a map
\[
r_{\zeta'}\colon
\{\text{quasi-periodic discrete conformal maps of period }\per\}
\longrightarrow
\mathcal P_n.
\]
\begin{definition}
    The \emph{solution map} is defined by
\[
\mathcal S
=
r_{\zeta'}\circ r_\zeta^{-1}
\colon
\mathcal P_n
\dashrightarrow
\mathcal P_n.
\]
\end{definition}
In other words, \(\mathcal S\) sends the restriction of a
quasi-periodic discrete conformal map to a zigzag to its restriction
to the adjacent parallel zigzag. The map \(\mathcal S\) is birational,
since both \(\mathcal S\) and its inverse are obtained by iteratively
solving the cross-ratio equation.

\medskip

We now show that the solution map $\S$ can be realized as a Coxeter element
of a suitably defined Coxeter group acting on quasi-periodic polygons. Observe that the passage from a zigzag to an adjacent parallel zigzag can be decomposed into a sequence of elementary local moves. Each such move replaces a corner of type $\ulcorner$ by a corner of type $\lrcorner$. In order to preserve periodicity, this move must be performed simultaneously at all corners related by translation by the period; see Figure \ref{fig:zzfolding}.
Accordingly, the shift map $\mathcal S$ can be decomposed into a sequence of elementary transformations. Each such transformation determines the value of a discrete conformal map at the fourth vertex of an elementary quadrilateral from its values at the other three vertices by means of the cross-ratio equation. We call these elementary transformations \emph{foldings}.

\begin{figure}[!hbt]
\centering
\begin{tikzpicture}[>=latex]

\begin{scope}

\fill[gray!20] (0,0) rectangle (0.5,0.5);
\fill[gray!20] (1.5,1) rectangle (2,1.5);

\draw[step=0.5cm,color=gray!40] (0,0) grid (3,2);

\draw[very thick]
(0,0) -- (0,0.5) -- (1,0.5) -- (1,1) -- (1.5,1)
-- (1.5,1.5) -- (2.5,1.5) -- (2.5,2) -- (3,2);

\end{scope}

\draw[->,thick] (3.8,1) -- (5.2,1);

\begin{scope}[shift={(6,0)}]

\fill[gray!20] (0,0) rectangle (0.5,0.5);
\fill[gray!20] (1.5,1) rectangle (2,1.5);

\draw[step=0.5cm,color=gray!40] (0,0) grid (3,2);

\draw[very thick]
(0,0) -- (0.5,0) -- (0.5,0.5) -- (1,0.5) -- (1,1)
-- (1.5,1) -- (2,1) -- (2,1.5) -- (2.5,1.5) -- (2.5,2) -- (3,2);

\end{scope}

\end{tikzpicture}
\caption{An elementary move on a $(3,2)$-periodic zigzag. The move is performed simultaneously in all squares related by the period.}
\label{fig:zzfolding}
\end{figure}
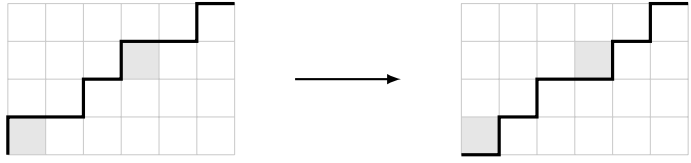

To define foldings, we regard \(\mathbb C\) as the boundary at infinity
of hyperbolic \(3\)-space \(\mathbb H^3\) (with one point removed).
Accordingly, a quasi-periodic polygon in \(\mathbb C\) may be viewed as
an ideal polygon in \(\mathbb H^3\), that is, a polygon whose vertices
lie at infinity. Given a quasi-periodic polygon \((p_i)\), its
\emph{folding} at the \(j\)th vertex is defined by reflecting the vertex
\(p_j\) in the hyperbolic geodesic joining \(p_{j-1}\) and
\(p_{j+1}\); see Figure~\ref{fig:folding}. Since the reflection depends
only on the triple
\(
p_{j-1},p_j,p_{j+1}, 
\)
it extends uniquely to all vertices congruent to \(j\) modulo \(n\),
producing another quasi-periodic \(n\)-gon. This leads to the following
formal definition.

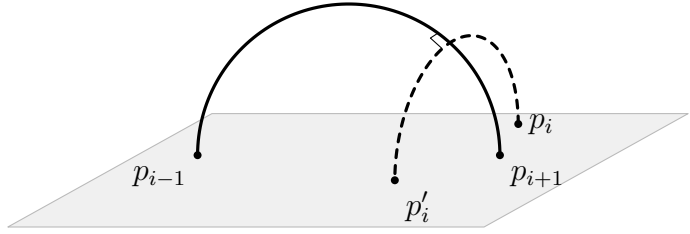
\begin{figure}[ht]
\centering

\begin{tikzpicture}[
  scale=1,
  x={(1cm,0cm)},
  y={(0.45cm,0.25cm)},
  z={(0cm,1cm)},
  line cap=round,
  line join=round
]
\newcommand{\TangentRightAngle}[4][3mm]{%
  \coordinate (raSolid)  at ($(#2)!#1!(#3)$);
  \coordinate (raDashed) at ($(#2)!#1!(#4)$);
  \coordinate (raCorner) at ($(raSolid)+(raDashed)-(#2)$);
  \draw (raSolid) -- (raCorner) -- (raDashed);
}

\pgfmathsetmacro{\a}{2.0}      

\pgfmathsetmacro{\Bx}{1.50}    
\pgfmathsetmacro{\By}{1.65}    

\pgfmathsetmacro{\s}{\Bx*\Bx + \By*\By}

\pgfmathsetmacro{\p}{2*\a*\a*\Bx/(\s + \a*\a)}
\pgfmathsetmacro{\h}{sqrt(\a*\a - \p*\p)}

\pgfmathsetmacro{\qx}{\Bx - \p}
\pgfmathsetmacro{\qy}{\By}
\pgfmathsetmacro{\q}{sqrt(\qx*\qx + \qy*\qy)}
\pgfmathsetmacro{\dx}{\qx/\q}
\pgfmathsetmacro{\dy}{\qy/\q}

\pgfmathsetmacro{\c}{(\q*\q - \h*\h)/(2*\q)}
\pgfmathsetmacro{\R}{sqrt(\c*\c + \h*\h)}

\pgfmathsetmacro{\Bpx}{\p + (\c - \R)*\dx}
\pgfmathsetmacro{\Bpy}{(\c - \R)*\dy}

\pgfmathsetmacro{\Mx}{\p + \c*\dx}
\pgfmathsetmacro{\My}{\c*\dy}

\fill[gray!12] (-2.8,-3.8,0) -- (3.5,-3.8,0) -- (3.5,2.2,0) -- (-2.8,2.2,0) -- cycle;
\draw[gray!55] (-2.8,-3.8,0) -- (3.5,-3.8,0) -- (3.5,2.2,0) -- (-2.8,2.2,0) -- cycle;


\coordinate (A)  at (-\a,0,0);
\coordinate (C)  at ( \a,0,0);
\coordinate (B)  at (\Bx,\By,0);
\coordinate (Bp) at (\Bpx,\Bpy,0);
\coordinate (P)  at (\p,0,\h);


\draw[very thick,black!70!black]
  plot[domain=0:180,samples=100,variable=\t]
  ({\a*cos(\t)},0,{\a*sin(\t)});

\draw[very thick,dashed,black!75!black]
  plot[domain=0:180,samples=100,variable=\t]
  ({\p + (\c + \R*cos(\t))*\dx},
   {(\c + \R*cos(\t))*\dy},
   {\R*sin(\t)});


\pgfmathsetmacro{\tsX}{-\h/\a}
\pgfmathsetmacro{\tsY}{0}
\pgfmathsetmacro{\tsZ}{\p/\a}

\pgfmathsetmacro{\tdX}{-(\h/\R)*\dx}
\pgfmathsetmacro{\tdY}{-(\h/\R)*\dy}
\pgfmathsetmacro{\tdZ}{-\c/\R}

\pgfmathsetmacro{\ra}{0.2}

\coordinate (RA1) at ({\p+\ra*\tsX},{0+\ra*\tsY},{\h+\ra*\tsZ});
\coordinate (RA2) at ({\p+\ra*\tdX},{0+\ra*\tdY},{\h+\ra*\tdZ});
\coordinate (RA3) at ({\p+\ra*\tsX+\ra*\tdX},
                     {0+\ra*\tsY+\ra*\tdY},
                     {\h+\ra*\tsZ+\ra*\tdZ});

\draw[thin] (RA2) -- (RA3) -- (RA1);

\fill (A) circle (1.5pt) node[below left] {$p_{i-1}$};
\fill (C) circle (1.5pt) node[below right] {$p_{i+1}$};
\fill (B) circle (1.5pt) node[right] {$p_i$};
\fill (Bp) circle (1.5pt) node[below right] {$p_i'$};


\end{tikzpicture}

\caption{Folding at $p_i$ for an ideal polygon in $\mathbb{H}^3$. }\label{fig:folding}

\end{figure}

\begin{definition}
The folding of a quasi-periodic \(n\)-gon \((p_i)\) at the
\(j\)th vertex is the polygon \((p_i')\) defined by
\[
p_k' = p_k \quad \text{for } k \not\equiv j \pmod n,
\]
and, for all \(k \equiv j \pmod n\), by the condition
\[
[p_{k-1}, p_k, p_{k+1}, p_k'] = -1,
\]
where \([\cdot,\cdot,\cdot,\cdot]\) denotes the cross-ratio
\eqref{eq:cross-ratio}.
\end{definition}

The resulting polygon is again quasi-periodic with the same monodromy.
Thus, folding at the \(j\)th vertex defines a birational self-map
\[
F_j\colon\mathcal P_n\dashrightarrow\mathcal P_n.
\]

\begin{remark}
We use the term \emph{folding} by analogy with the Euclidean setting,
where a folding is the reflection of a vertex \(p_j\) of a polygon
\((p_i\in\mathbb R^2)\) in the diagonal joining its neighboring
vertices \(p_{j-1}\) and \(p_{j+1}\); see
\cite{Izmestiev2023}. For quadrilaterals, the composition of foldings
at two adjacent vertices is an integrable system also known as the
\emph{Darboux transformation}; see
\cite{DragovicRadnovic2025}.
\end{remark}
As we pass from one zigzag to another by a sequence of elementary moves as in Figure~\ref{fig:zzfolding}, the corresponding transformation of the initial data is given by a composition of the folding maps $F_j$. In particular, the solution map $\mathcal S$, which corresponds to passing from a zigzag to the adjacent parallel zigzag, can be written as a composition of foldings. This composition can be described explicitly in terms of a suitable Coxeter group. We therefore begin with a brief review of the necessary Coxeter-group background.
\begin{definition}
    A \emph{Coxeter group} is a group generated by involutions
$\sigma_1,\dots,\sigma_n$ subject to relations of the form
\[
(\sigma_i\sigma_j)^{m_{ij}}=1,
\]
where $m_{ij}\in\{2,3,\dots,\infty\}$ and $m_{ij}=\infty$ means that no
relation is imposed. 
\end{definition}  A {Coxeter group} can be depicted by its \emph{Coxeter diagram}.
The vertices of the Coxeter diagram correspond to the generators
$\sigma_i$. Two vertices are connected by an edge if and only if the
corresponding generators $\sigma_i$ and $\sigma_j$ do not commute, i.e.,
$m_{ij}>2$. Additionally, edges are labeled by the integers $m_{ij}$. 

\begin{definition}
    Let $W=\langle \sigma_1,\dots,\sigma_n\rangle$ be a Coxeter group.
A \emph{Coxeter element} of $W$ is the product of all generators
taken in some order:
\[
c=\sigma_{\pi(1)}\sigma_{\pi(2)}\cdots\sigma_{\pi(n)},
\]
where $\pi$ is a permutation of $\{1,\dots,n\}$.\end{definition} Since commuting generators can be interchanged without affecting the
resulting product, Coxeter elements are naturally parametrized by acyclic
orientations of the Coxeter diagram. Given a Coxeter element
$c=\sigma_{\pi(1)}\cdots\sigma_{\pi(n)}$, one obtains an orientation of the
Coxeter diagram by directing the edge joining $\sigma_i$ and $\sigma_j$
(if such an edge exists) from $\sigma_i$ to $\sigma_j$ whenever
$\sigma_j$ precedes $\sigma_i$ in the product. Conversely, every acyclic
orientation arises in this way; see, e.g., \cite[Lemma 4.1]{Coxeterquiver_Tomas}.


\medskip

Returning to our setting, consider the group $\FG_n$ generated by
the foldings $F_1,\dots,F_n$ acting on the space of quasi-periodic $n$-gons.
Each folding is an involution, and foldings at the $i$th and $j$th
vertices commute unless $i$ and $j$ are adjacent in the cyclic order on
$\{1,\dots,n\}$. Thus, $\FG_n$ is a Coxeter group whose Coxeter
diagram is the necklace graph shown in Figure~\ref{fig:Coxeter-pentagon}
(all edge labels are equal to $\infty$ and are therefore omitted).
Consequently, Coxeter elements of $\FG_n$ are in bijection with
acyclic orientations of the necklace graph; see
Figure~\ref{fig:Coxeter-oriented}.

\begin{figure}[ht]
\centering

\begin{subfigure}{0.4\linewidth}
\centering
\begin{tikzpicture}[scale=1.3,
  edge/.style={thick,shorten <=3mm,shorten >=3mm},
  curved/.style={bend #1=30,looseness=0.75}]
\tikzset{dot/.style={circle,inner sep=5pt,draw,name=#1}}

\coordinate (c0) at (0.5,2.0);
\coordinate (c1) at (1.5,2.0);
\coordinate (c2) at (2.5,2.0);
\coordinate (c3) at (3.5,2.0);
\coordinate (c4) at (4.5,2.0);

\node[dot=$1$] at (c0){}; \node at (c0) {\footnotesize$1$};
\node[dot=$1$] at (c1){}; \node at (c1) {\footnotesize$2$};
\node[dot=$2$] at (c2){}; \node at (c2) {\footnotesize$3$};
\node[dot=$3$] at (c3){}; \node at (c3) {\footnotesize$4$};
\node[dot=$4$] at (c4){}; \node at (c4) {\footnotesize$5$};

\draw[edge] (c4) to[curved=right] (c0);
\draw[edge] (c1) -- (c0);
\draw[edge] (c1) -- (c2);
\draw[edge] (c2) -- (c3);
\draw[edge] (c4) -- (c3);

\end{tikzpicture}
\caption{The Coxeter diagram of the group of foldings of pentagons.}
\label{fig:Coxeter-pentagon}
\end{subfigure}
\qquad
\begin{subfigure}{0.4\linewidth}
\centering
\begin{tikzpicture}[scale=1.3,
  myarrow/.style={->,thick,shorten <=3mm,shorten >=3mm},
  curved/.style={bend #1=30,looseness=0.75}]
\tikzset{dot/.style={circle,inner sep=5pt,draw,name=#1}}

\coordinate (c0) at (0.5,2.0);
\coordinate (c1) at (1.5,2.0);
\coordinate (c2) at (2.5,2.0);
\coordinate (c3) at (3.5,2.0);
\coordinate (c4) at (4.5,2.0);

\node[dot=$1$] at (c0){}; \node at (c0) {\footnotesize$1$};
\node[dot=$1$] at (c1){}; \node at (c1) {\footnotesize$2$};
\node[dot=$2$] at (c2){}; \node at (c2) {\footnotesize$3$};
\node[dot=$3$] at (c3){}; \node at (c3) {\footnotesize$4$};
\node[dot=$4$] at (c4){}; \node at (c4) {\footnotesize$5$};

\draw[myarrow] (c4) to[curved=right] (c0);
\draw[myarrow] (c1) -- (c0);
\draw[myarrow] (c1) -- (c2);
\draw[myarrow] (c2) -- (c3);
\draw[myarrow] (c4) -- (c3);

\end{tikzpicture}
\caption{The oriented Coxeter diagram corresponding to the folding sequence
$F_4\circ F_1\circ F_3\circ F_5\circ F_2$.}
\label{fig:Coxeter-oriented}
\end{subfigure}

\caption{}
\label{fig:Coxeter}
\end{figure}

On the other hand, acyclic orientations of the necklace graph with $n$
vertices are naturally in bijection with translation classes of
$n$-periodic zigzags in the square lattice $\mathbb Z^2$. Given an orientation, we
declare the $j$th edge of the zigzag to be horizontal if the edge joining
the vertices $j$ and $j+1$ of the necklace graph is oriented towards
$j+1$, and vertical otherwise. Conversely, every zigzag determines an
acyclic orientation of the necklace graph. For example, the orientation
shown in Figure~\ref{fig:Coxeter-oriented} corresponds to the zigzag shown
in Figure~\ref{fig:zig}.

To summarize, we have the following bijections:
\[
\left\{
\parbox{3.2cm}{\centering
Coxeter elements of\\
the folding group $\FG_n$}
\right\}
\longleftrightarrow
\left\{
\parbox{3.2cm}{\centering
acyclic orientations of the\\
necklace graph on $n$ vertices}
\right\}
\longleftrightarrow
\left\{
\parbox{3.2cm}{\centering
translation classes of\\
$n$-periodic zigzags}
\right\}.
\]

\begin{proposition}\label{prop:shiftfold}
Let $\zeta$ be an $n$-periodic zigzag. Under the correspondence between
translation classes of $n$-periodic zigzags and Coxeter elements of the
folding group $\FG_n$, the solution map
\[
\mathcal S \colon \mathcal P_n \dashrightarrow \mathcal P_n
\]
associated with $\zeta$ is precisely the Coxeter element corresponding to
$\zeta$.
\end{proposition}

\begin{proof}
As explained above, the solution map is obtained by a sequence of foldings,
each corresponding to an elementary move on zigzags of the form shown in
Figure~\ref{fig:zzfolding}. Collectively, these moves transform a zigzag
$\zeta$ into its adjacent parallel translate $\zeta'$.

To pass from $\zeta$ to $\zeta'$, every vertex of $\zeta$ must be shifted
once in the southeast direction. Consequently, the solution map is a
composition of foldings in which each generator $F_i$ appears exactly
once. In other words, it is represented by a Coxeter element.

It remains to determine the order of the factors. Consider two adjacent
vertices $i$ and $i+1$ of the zigzag. If the edge joining them is
horizontal, then the move at the $i$th vertex must be performed before the
move at the $(i+1)$st vertex. Since compositions are read from right to
left, the corresponding Coxeter word contains $F_{i+1}$ to the left
of~$F_i$.

Similarly, if the edge joining $i$ and $i+1$ is vertical, then the move at
the $(i+1)$st vertex must be performed before the move at the $i$th
vertex, and hence the corresponding Coxeter word contains $F_i$ to the
left of~$F_{i+1}$.

Thus
\[
\mathcal S = F_{i_1}\circ\cdots\circ F_{i_n},
\]
where $(i_1,\dots,i_n)$ is a permutation of $(1,\dots,n)$, and $F_i$
appears to the left of $F_{i+1}$ if and only if the edge of $\zeta$
joining the vertices $i$ and $i+1$ is vertical. In terms of the necklace
graph, the latter condition means precisely that the edge joining the
vertices $i$ and $i+1$ is oriented towards $i$. Therefore, the order of
the generators in the Coxeter word representing $\mathcal S$ is exactly
the one prescribed by the orientation corresponding to $\zeta$. Hence
$\mathcal S$ is the Coxeter element associated with that orientation, or,
equivalently, with the translation class of the zigzag~$\zeta$.
\end{proof}




\subsection{Discrete conformality as zero curvature}
In this section, we show that discrete conformality of a map
\(
f\colon\mathbb Z^2\to\mathbb{CP}^1
\)
is equivalent to flatness of a natural connection on
\(\mathbb Z^2\) associated with \(f\). This provides a
\emph{zero-curvature representation} of the cross-ratio equation. Such
a representation is well known to exist as a consequence of the
three-dimensional consistency of the cross-ratio equation. Here we give
a concise direct derivation that does not rely on three-dimensional
consistency.

Moreover, we define \emph{two} flat connections associated with a discrete
conformal map. One of them, denoted by \(\Pi\), takes values in the loop
group of \(\PGL_2(\mathbb C)\). The other, denoted by \(\pi\), takes
values in the additive group of the Lie algebra
\(\mathfrak{pgl}_2(\mathbb C)\) and may be viewed as the linearization
of \(\Pi\). The former will be used in the following sections to
construct first integrals of the solution map, while the latter serves
as an auxiliary tool.
\begin{definition}
Let $\Gamma$ be a graph, and let $G$ be a group. A \emph{$G$-connection}
on $\Gamma$ is a map
\(
e \mapsto \phi_e
\)
from the set of oriented edges of $\Gamma$ to $G$ such that
\[
\phi_{\bar e}=\phi_e^{-1},
\]
where $\bar e$ denotes the edge $e$ with reversed orientation.

Given a based loop \(e_1,\dots,e_n\) in \(\Gamma\), that is, a sequence of
oriented edges such that the head \(h(e_i)\) of \(e_i\) coincides with
the tail \(t(e_{i+1})\) of \(e_{i+1}\) (with indices understood
cyclically), the corresponding \emph{holonomy} is the product
\[
\phi_{e_1}\cdots\phi_{e_n}.
\]
\end{definition}

\begin{definition}
Suppose that the graph $\Gamma$ is embedded in a $2$-manifold $\Sigma$
so that its \emph{faces}, i.e., the connected components of
$\Sigma\setminus\Gamma$, are topological disks. A $G$-connection $\phi$
on $\Gamma$ is called \emph{flat} if its holonomy along the boundary of
every face of $\Gamma$ is the identity. (Note that the holonomy around a loop depends on the choice of base
point; however, changing the base point conjugates the holonomy, so
whether it is the identity is independent of the choice of base point.)
\end{definition}
We then have the following standard result.

\begin{proposition}
For a flat connection, the holonomy along a based loop depends only on
its homotopy class in $\pi_1(\Sigma)$.
\end{proposition}

 

Returning to the setting of discrete conformal maps, let
\[
f\colon\mathbb Z^2\to\mathbb{CP}^1
\]
be a map sending adjacent lattice points to distinct points. We now
associate to \(f\) two connections on \(\mathbb Z^2\), denoted by
\(\pi\) and \(\Pi\). 

The first connection, \(\pi\), takes values in the additive group
\(\mathfrak{pgl}_2(\mathbb C)\), that is, the vector space of
\(2\times2\) matrices modulo scalar matrices. For an oriented edge
\(e\), let
\[
a:=f_{t(e)},
\qquad
b:=f_{h(e)},
\]
and define
\[
\pi_e:=\eps_e\proj_a^b,
\]
where \(\eps_e=1\) if \(e\) is horizontal and
\(\eps_e=-1\) if \(e\) is vertical, and where
\(\proj_a^b\) denotes the projection onto the line
\(a\subset\mathbb C^2\) along the line
\(b\subset\mathbb C^2\). Since
\[
\proj_a^b+\proj_b^a=\id=0
\]
in \(\mathfrak{pgl}_2(\mathbb C)\), we have
\[
\pi_{\bar e}=-\pi_e,
\]
so \(\pi\) is indeed a connection.

The second connection, \(\Pi\), takes values in
\(\PGL_2(\mathbb C(t))\), where \(\mathbb C(t)\) denotes the field of
rational functions in an indeterminate \(t\) (the
\emph{spectral parameter}). It is defined by
\begin{equation}\label{eq:Piconn}
\Pi_e(t):=\id-\eps_e t\,\proj_a^b.
\end{equation}
Since
\[
\Pi_{\bar e}(t)
=
\id-\eps_e t\,\proj_b^a,
\]
and
\[
\proj_a^b+\proj_b^a=\id,
\qquad
\proj_a^b\proj_b^a
=
\proj_b^a\proj_a^b
=
0,
\]
we obtain
\[
\Pi_e(t)\,\Pi_{\bar e}(t)
=
(1-\eps_e t)\id.
\]
Thus \(\Pi_e(t)\Pi_{\bar e}(t)\) is the identity in
\(\PGL_2(\mathbb C(t))\), so
\[
\Pi_{\bar e}(t)=\Pi_e(t)^{-1}.
\]
Hence \(\Pi\) is also a connection.






\begin{proposition}
Let
\[
f\colon\mathbb Z^2\to\mathbb{CP}^1
\]
be a map sending adjacent lattice points to distinct points. Then the
following are equivalent:
\begin{enumerate}
\item \(f\) is discrete conformal;
\item the connection \(\pi\) is flat;
\item the connection \(\Pi\) is flat.
\end{enumerate}
\end{proposition}

\begin{proof}
Let \(a,b,c,d\) be the images under \(f\) of the vertices of an
elementary square of \(\mathbb Z^2\), listed counterclockwise. To prove
the equivalence of (1) and (2), it suffices to show that
\[
[a,b,c,d]=-1
\quad\Longleftrightarrow\quad
P(a,b,c,d)=0,
\]
where
\[
P(a,b,c,d)
:=
\proj_a^b-\proj_b^c+\proj_c^d-\proj_d^a.
\]
(Strictly speaking, flatness of \(\pi\) requires \(P(a,b,c,d)\) to vanish
only modulo scalar matrices. However, \(P(a,b,c,d)\) has trace zero and
is therefore scalar if and only if it is zero.)

The condition
\(
[a,b,c,d]=-1
\)
admits the following interpretation. Consider the restrictions to the
line \(a\) of the projections onto \(c\) along \(b\) and \(d\),
respectively:
\[
\left.\proj_c^b\right|_a,
\,
\left.\proj_c^d\right|_a
\colon
a\to c.
\]
Since \(a\) and \(c\) are one-dimensional, these maps differ by a
scalar, namely the cross-ratio \([a,b,c,d]\). Thus
\[
[a,b,c,d]=-1
\quad\Longleftrightarrow\quad
\left.(\proj_c^b+\proj_c^d)\right|_a=0,
\]
see Figure \ref{fig:crossratio-minus-one}.

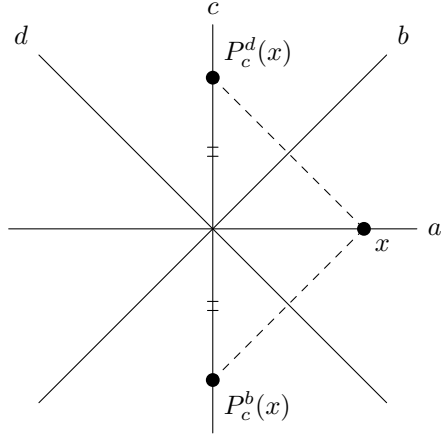
\begin{figure}[!ht]
\centering

\begin{tikzpicture}[
  scale=2,
  line cap=round,
  line join=round
]

\coordinate (O)  at (0,0);
\coordinate (x)  at (1,0);
\coordinate (Pbd) at (0,-1);
\coordinate (Pdd) at (0,1);

\draw[] (-1.35,0) -- (1.35,0) node[right] {\footnotesize$a$};
\draw[] (0,-1.35) -- (0,1.35) node[above] {\footnotesize$c$};

\draw[] (-1.15,-1.15) -- (1.15,1.15) node[above right] {\footnotesize$b$};
\draw[] (1.15,-1.15) -- (-1.15,1.15) node[above left] {\footnotesize$d$};

\draw[dashed,] (x) -- (Pbd);
\draw[dashed,] (x) -- (Pdd);

\fill (x) circle (1.3pt) node[below right] {\footnotesize $x$};
\fill (Pbd) circle (1.3pt) node[below right] {\footnotesize$P_c^b(x)$};
\fill (Pdd) circle (1.3pt) node[above right] {\footnotesize$P_c^d(x)$};

\draw[] (-0.035,0.48) -- (0.035,0.48);
\draw[] (-0.035,0.54) -- (0.035,0.54);

\draw[] (-0.035,-0.48) -- (0.035,-0.48);
\draw[] (-0.035,-0.54) -- (0.035,-0.54);

\end{tikzpicture}

\caption{An interpretation of the condition \([a,b,c,d]=-1\).}
\label{fig:crossratio-minus-one}

\end{figure}

Assume first that
\(
[a,b,c,d]=-1.
\)
Restricting \(P(a,b,c,d)\) to \(a\), we obtain
\[
\left.P(a,b,c,d)\vphantom{\id-\proj_b^c+\proj_c^d)}\right|_a
=
\left.(\id-\proj_b^c+\proj_c^d)\right|_a
=
\left.(\proj_c^b+\proj_c^d)\right|_a
=
0.
\]
Since
\(
[b,c,d,a]=-1,
\)
the same argument shows that \(P(a,b,c,d)\) vanishes on \(b\). As \(a\)
and \(b\) are distinct lines, it follows that
\(
P(a,b,c,d)=0.
\)
Conversely, if
\(
P(a,b,c,d)=0,
\)
then restricting to \(a\) gives
\[
\left.(\proj_c^b+\proj_c^d)\right|_a=0,
\]
and hence
\(
[a,b,c,d]=-1. 
\) 

Now we prove equivalence of (2) and (3). Flatness of \(\Pi\) around the same elementary square is
equivalent to
\[
(\id-t\proj_a^b)(\id+t\proj_b^c)
=
(\id+t\proj_a^d)(\id-t\proj_d^c).
\]
Since
\[
\proj_a^b\proj_b^c=0,
\qquad
\proj_a^d\proj_d^c=0,
\]
this is equivalent to
\[
\proj_b^c-\proj_a^b
=
\proj_a^d-\proj_d^c.
\]
Using
\[
\proj_a^d+\proj_d^a=\id,
\qquad
\proj_d^c+\proj_c^d=\id,
\]
we rewrite this as
\[
\proj_a^b-\proj_b^c+\proj_c^d-\proj_d^a=0,
\]
which is precisely the flatness condition for \(\pi\). Hence
\(\pi\) is flat if and only if \(\Pi\) is flat.
\end{proof}

\begin{remark}
The connection \(\Pi\) admits the following geometric interpretation.
Let \(a,b\in\CP^1\) be distinct. For every \(t\in\C\), there is a unique
projective transformation \(g_t\) satisfying
\[
[a,b,x,g_t(x)]=t
\]
for every \(x\in\CP^1\) for which the cross-ratio is defined.
Explicitly, \(g_t=\id-tP_a^b\). Thus, if an oriented edge \(e\) has tail
\(a\) and head \(b\), then \(\Pi_e(t)=g_{\pm t}\), where the sign is
positive for horizontal edges and negative for vertical ones.

Hence the zero-curvature representation of the cross-ratio equation is
itself given by the same equation, with the constant value \(-1\)
replaced by the spectral parameter \(\pm t\). This is a characteristic
feature of three-dimensionally consistent quad-equations.
\end{remark}

\subsection{A periodic zero-curvature representation}

In what follows, we will use a flat connection to construct first
integrals of the solution map. This construction requires the
connection to be periodic. The connection \(\Pi\) constructed in the
previous section is generally not periodic: it is periodic only when
the discrete conformal map itself is periodic. We therefore replace it
by a gauge-equivalent flat connection \(Z\), which is periodic for
every quasi-periodic discrete conformal map.


\begin{definition}
Two $G$-connections $\phi$ and $\psi$ on a graph $\Gamma$ are
called \emph{gauge-equivalent} if there exists a $G$-valued function
$\theta$ on the vertex set of $\Gamma$ such that
\[
\psi_e
=
\theta_{t(e)}^{-1}\phi_e\theta_{h(e)}
\]
for every oriented edge $e$ of $\Gamma$.
\end{definition}

Gauge-equivalent connections have conjugate holonomies. In particular, a
connection that is gauge-equivalent to a flat connection is itself flat.

\begin{proposition}\label{prop:Z-flat}
\begin{enumerate}
\item Let
\(
f\colon \mathbb Z^2\to\C
\)
be a map sending adjacent lattice points to distinct points. Define a
\(\PGL_2(\mathbb C(t))\)-valued connection on \(\mathbb Z^2\) by
setting, for every oriented edge \(e\),
\begin{equation}\label{eq:Z-edge}
Z_e
:=
\begin{pmatrix}
1 & \varepsilon_e t / z_e\\[0.6ex]
z_e & 1
\end{pmatrix},
\qquad
z_e:=f_{h(e)}-f_{t(e)}.
\end{equation}
Then \(f\) is discrete conformal if and only if \(Z\) is flat.

\item The assignment \(f\mapsto Z\) induces a bijection between
discrete conformal maps modulo translations and flat connections of the
form \eqref{eq:Z-edge}.
\end{enumerate}
\end{proposition}

\begin{proof}
Let
\[
\Theta_{i,j}
=
\begin{pmatrix}
f_{i,j} & 1\\
1 & 0
\end{pmatrix}.
\]
Then, for every oriented edge \(e\), we have
\[
Z_e
=
\Theta_{t(e)}^{-1}\Pi_e(t)\Theta_{h(e)},
\]
where the connection $\Pi$ is given by \eqref{eq:Piconn}. Thus \(Z\) is gauge-equivalent to \(\Pi\). Since \(\Pi\) is flat if and
only if \(f\) satisfies the cross-ratio equation, the first statement
follows.

For the second statement, it suffices to show that every flat
connection of the form~\eqref{eq:Z-edge} arises from a map
\(f\colon\mathbb Z^2\to\C\).
Such a map exists if and only if the increments \(z_e\) sum to zero
around every elementary quadrilateral. Let
\(e_1,e_2,e_3,e_4\) be the oriented boundary edges of such a
quadrilateral. Then
\[
Z_{e_1}Z_{e_2}Z_{e_3}Z_{e_4}
=
\begin{pmatrix}
* & *\\
z_{e_1}+z_{e_2}+z_{e_3}+z_{e_4}+O(t) & *
\end{pmatrix},
\]
where \(O(t)\) denotes terms divisible by \(t\). Since the connection
is flat, this product is the identity in
\(\PGL_2(\mathbb C(t))\), so
\[
z_{e_1}+z_{e_2}+z_{e_3}+z_{e_4}=0.
\]
Hence the increments \(z_e\) integrate to a map
\(f\colon\mathbb Z^2\to\C\), unique up to an additive constant.
\end{proof}

The resulting zero-curvature representation, given by the connection
\(Z\), coincides with the one obtained from the three-dimensional
consistency of the cross-ratio equation~\cite[p.~227]{BobenkoSuris2008}.

\begin{remark}[Duality] \label{rem:duality}
Given a flat \(G\)-connection on a graph, any automorphism of \(G\)
produces another flat connection. In the case
\(G=\PGL_2\), there is an automorphism represented on matrices by
\(
A\longmapsto\det(A)\,(A^{-1})^T,
\)
that is, by taking the cofactor matrix. Applied to the connection
\eqref{eq:Z-edge}, this gives
\[
Z_e
\longmapsto
\begin{pmatrix}
1 & -z_e\\[1ex]
-\varepsilon_e t / z_e & 1
\end{pmatrix}.
\]
Conjugating further by the diagonal matrix
\(
\operatorname{diag}(t,-1),
\)
and setting
\(
z_e^*:=\varepsilon_e/z_e,
\)
we obtain the connection
\[
Z_e^*
=
\begin{pmatrix}
1 & \varepsilon_e t / z_e^*\\[1ex]
z_e^* & 1
\end{pmatrix}.
\]
Since this is again of the form~\eqref{eq:Z-edge},
Proposition~\ref{prop:Z-flat} implies that the flat connection \(Z^*\)
corresponds to a discrete conformal map \(f^*\), unique up to
translation. The map \(f^*\) is called the \emph{dual} of \(f\). The edge increments of the dual are precisely the variables \(z_e^*\),
which agrees, up to complex conjugation, with the notion of duality
introduced in~\cite{bobenko1996discrete,bobenko1999discrete}. Since duality depends
only on the edge increments, it preserves quasi-periodicity and hence
defines an involution on quasi-periodic discrete conformal maps modulo
translations. Moreover,
\[
(af)^*=a^{-1}f^*+\mathrm{const},
\]
so duality descends to the quotient by affine transformations.
\end{remark}

\subsection{The first integrals}\label{sec:fi}
In this section, we construct the first integrals \(I_j\) of the solution
map \(\mathcal S\) as the coefficients of the trace of the holonomy
matrix associated with the flat connection~\eqref{eq:Z-edge}. We then
show that they descend to the quotient by affine transformations.

Let \(f\) be a quasi-periodic discrete conformal map with period
\(
\per=(\hp,\vp).
\)
By construction, the connection \(Z\) given by \eqref{eq:Z-edge} is \(\per\)-periodic. Hence it
descends to a flat connection on the quotient graph
\(
\mathbb Z^2/\mathbb Z\per,
\)
viewed as a graph embedded in a cylinder. Since the connection is flat, the conjugacy class of its holonomy along
a loop depends only on the homotopy class of the loop, or equivalently,
on its winding number around the cylinder. In particular, the holonomy
is, up to conjugation, the same along every \(\per\)-periodic zigzag
\(
\zeta\colon\mathbb Z\to\mathbb Z^2.
\)
Let
\(
p_i:=f_{\zeta(i)}
\)
be the corresponding quasi-periodic \(n\)-gon, where
\(
n=\hp+\vp.
\)
Then the holonomy along \(\zeta\) is given by
\[
\mathbf Z(t):=Z_1(t)\cdots Z_n(t),
\]
where
\begin{equation}\label{eq:zmatrix}
Z_i(t):=
\begin{pmatrix}
1 & \varepsilon_i t / z_i\\[1ex]
z_i & 1
\end{pmatrix},
\end{equation}
\(z_i:=p_{i+1}-p_i\), and
\[
\varepsilon_i=
\begin{cases}
1,&\text{if }(\zeta_i,\zeta_{i+1})\text{ is horizontal},\\
-1,&\text{if }(\zeta_i,\zeta_{i+1})\text{ is vertical}.
\end{cases}
\]

\begin{proposition}
The conjugacy class of \(\mathbf Z(t)\) in
\(\GL_2(\mathbb C(t))\) is invariant under the solution map
\(\mathcal S\).
\end{proposition}

\begin{proof}
The solution map replaces the initial data on a zigzag \(\zeta\) by the
initial data on the adjacent parallel zigzag \(\zeta'\). Accordingly,
it replaces the holonomy matrix \(\mathbf Z(t)\) by the holonomy matrix
\(\mathbf Z'(t)\) associated with \(\zeta'\). Since \(\zeta\) and
\(\zeta'\) are homotopic on the cylinder, \(\mathbf Z(t)\) and
\(\mathbf Z'(t)\) are conjugate in \(\PGL_2(\mathbb C(t))\).

Furthermore,
\(
\det Z_i(t)=1-\varepsilon_i t,
\)
and hence
\[
\det \mathbf Z(t)=(1-t)^{\hp}(1+t)^{\vp}.
\]
The same formula holds for \(\mathbf Z'(t)\). Therefore
\(\mathbf Z(t)\) and \(\mathbf Z'(t)\) have the same determinant.

Since \(\mathbf Z(t)\) and \(\mathbf Z'(t)\) are conjugate in
\(\PGL_2(\mathbb C(t))\), there exist
\(A(t)\in\GL_2(\mathbb C(t))\) and
\(\lambda(t)\in\mathbb C(t)^\times\) such that
\[
\mathbf Z'(t)
=
\lambda(t)\,A(t)\,\mathbf Z(t)\,A(t)^{-1}.
\]
Taking determinants yields
\(
\lambda(t)^2=1,
\)
so \(\lambda(t)=\pm1\).

Finally, each matrix \(Z_i(t)\) is lower unitriangular at
\(t=0\). Hence the same is true for both
\(\mathbf Z(t)\) and
\(\mathbf Z'(t)\). In particular, both have trace \(2\) at
\(t=0\). Therefore \(\mathbf Z'(t)\) cannot be conjugate to
\(-\mathbf Z(t)\). It follows that
\(\mathbf Z'(t)\) is conjugate to
\(\mathbf Z(t)\) in
\(\GL_2(\mathbb C(t))\).
\end{proof}

\begin{corollary}\label{cor:Ik}
Let
\[
\tr\mathbf Z(t)=\sum_{k\ge0} I_k t^k.
\]
Then the coefficients \(I_k\) are
\(\mathcal S\)-invariant functions on \(\mathcal P_n\).
\end{corollary}

\begin{proof}
The trace is invariant under conjugation in
\(\GL_2(\mathbb C(t))\).
\end{proof}

We now determine the number of nontrivial coefficients \(I_k\).
First, note that since \(\mathbf Z(0)\) is lower unitriangular, we have
\(I_0=2\). Furthermore, the following proposition shows that
\(I_k=0\) for \(k>\lfloor n/2\rfloor\).

\begin{proposition}\label{prop:Z-form}
We have
\[
\mathbf Z(t)=
\begin{pmatrix}
A(t)&B(t)\\
C(t)&D(t)
\end{pmatrix},
\]
where
\[
\deg A,\deg D\le\left\lfloor\frac n2\right\rfloor,
\qquad
\deg B\le\left\lfloor\frac{n+1}{2}\right\rfloor,
\qquad
\deg C\le\left\lfloor\frac{n-1}{2}\right\rfloor.
\]
\end{proposition}

\begin{proof}
The entries of each factor \(Z_i(t)\) have degrees
\[
Q:=
\begin{pmatrix}
0&1\\
0&0
\end{pmatrix}.
\]
Therefore the degrees of the entries of
\(
\mathbf Z(t)=Z_1(t)\cdots Z_n(t)
\)
are bounded by the max-plus power \(Q^{\otimes n}\). A direct induction
shows that
\[
Q^{\otimes n}
=
\begin{pmatrix}
\lfloor n/2\rfloor & \lfloor (n+1)/2\rfloor\\
\lfloor (n-1)/2\rfloor & \lfloor n/2\rfloor
\end{pmatrix},
\]
which proves the stated degree bounds. 
\end{proof}

Since \(I_0=2\), we obtain
\(
\lfloor n/2\rfloor
\)
\(\mathcal S\)-invariant functions on \(\mathcal P_n\), namely
\(
I_1,\dots,I_{\lfloor n/2\rfloor}.
\)
\begin{proposition}
The functions \(I_j\) descend to well-defined functions on the quotient
\(
\mathcal P_n/\Aff(\mathbb C).
\)
\end{proposition}

\begin{proof}
Under an affine transformation \(z\mapsto az+b\), all edge vectors
\(z_i\) are multiplied by the same scalar \(a\). Consequently, each
matrix \(Z_i(t)\) is conjugated by the diagonal matrix
\[
\begin{pmatrix}
1&0\\
0&a
\end{pmatrix}.
\]
Hence the holonomy matrix \(\mathbf Z(t)\) is conjugated by the same
matrix. Therefore its trace, and thus each coefficient \(I_j\), is
invariant under affine transformations.
\end{proof}
We now express the functions \(I_j\) in terms of the coordinates
\[
y_i:=\frac{z_i}{z_{i-1}}
\]
on \(\mathcal P_n/\Aff(\mathbb C)\). 
\begin{proposition}\label{prop:Iy}
Let
\begin{equation}\label{eq:ymatrix}
Y_i(t):=
\begin{pmatrix}
1&\varepsilon_{i-1} y_i t\\
1&y_i
\end{pmatrix},
\end{equation}
and
\[
\mathbf Y(t):=
Y_1(t)\cdots Y_n(t).
\]
Then the integrals
\(
I_1,\dots,I_{\lfloor n/2\rfloor}
\)
are precisely the coefficients of powers of \(t\) in
\(
\tr\mathbf Y(t).
\)
\end{proposition}

\begin{proof}
Let
\(
\Lambda_i:=
\mathrm{diag}(1, z_i). 
\)
Then
\[
Y_i(t)=
\Lambda_{i-1}^{-1}Z_{i-1}(t)\Lambda_i.
\]

Therefore,
\(
\mathbf Y(t)=
Y_1(t)\cdots Y_n(t)
\)
is conjugate to \(\mathbf Z(t) = Z_1(t)\cdots Z_n(t)\). Hence
\[
\tr\mathbf Y(t)=\tr\mathbf Z(t),
\]
and the result follows from the definition of the functions \(I_j\).
\end{proof}


\begin{remark}
The first integral \(I_1\) admits the following geometric
interpretation:
\[
I_1=\vp-\hp+\Delta\Delta^*,
\]
where $(\hp, \vp) $ is the period, \(\Delta\) is the monodromy of the quasi-periodic discrete
conformal map \(f\), and \(\Delta^*\) is the monodromy of its dual
\(f^*\); see Remark~\ref{rem:duality}. Although neither \(\Delta\) nor \(\Delta^*\) is invariant under affine
transformations \(f\mapsto af+b\), they transform as
\[
\Delta\mapsto a\Delta,
\qquad
\Delta^*\mapsto a^{-1}\Delta^*.
\]
 Consequently, the product
\(\Delta\Delta^*\) is a well-defined function on the quotient
\(\mathcal P_n/\Aff(\mathbb C)\). Its invariance under the solution map is immediate: the monodromy
\(\Delta\) is preserved by construction, while \(\Delta^*\) is preserved
because the solution map commutes with duality.
\end{remark}

\subsection{Independence of integrals}\label{sec:iqp}

In this section, we prove that the
\(\mathcal S\)-invariant functions
\(
I_1,\dots,I_{\lfloor n/2\rfloor}
\)
are algebraically independent as functions on
\(
\mathcal P_n/\Aff(\mathbb C).
\)
Equivalently, we prove the following statement.

\begin{proposition}\label{prop:I-dominant}
Fix an arbitrary sign pattern
\(
\varepsilon_1,\dots,\varepsilon_n\in\{\pm1\},
\)
and let \(Z_i(t)\) be the matrices defined by
\eqref{eq:Z-edge}. Let the functions $I_j$ be defined by
\[
\tr\bigl(Z_1(t)\cdots Z_n(t)\bigr)
=
2+\sum_{j=1}^{\lfloor n/2\rfloor}I_jt^j.
\]
Then the map
\[
(\C^\times)^n
\longrightarrow
\C^{\lfloor n/2\rfloor},
\qquad
(z_1,\dots,z_n)
\longmapsto
(I_1,\dots,I_{\lfloor n/2\rfloor})
\]
is dominant.
\end{proposition}

\begin{remark}
Dominance of a morphism to an affine space is equivalent to algebraic
independence of its components. Thus the functions
\(
I_1,\dots,I_{\lfloor n/2\rfloor}
\)
are algebraically independent as functions of
\(z_1,\dots,z_n\). Since they are invariant under affine
transformations, they descend to algebraically independent functions on
\(
\mathcal P_n/\Aff(\mathbb C).
\)
\end{remark}
\noindent Let
$$
Q(t):= (1-\varepsilon_1t)\cdots(1-\varepsilon_nt)
$$
and
\[
m:=\lfloor n/2\rfloor,\qquad
\overline m:=\lfloor(n+1)/2\rfloor,\qquad
\underline m:=\lfloor(n-1)/2\rfloor.
\]
Define
\[
\mathcal Z_n:=
\left\{
\left.
\begin{pmatrix}
1+a_1t+\cdots+a_mt^m &
b_1t+\cdots+b_{\overline m}t^{\overline m}\\[1ex]
c_0+c_1t+\cdots+c_{\underline m}t^{\underline m} &
1+d_1t+\cdots+d_mt^m
\end{pmatrix}
\;\right|\;
\det=Q(t)
\right\}.
\]
To prove Proposition~\ref{prop:I-dominant}, we factor the map
\[
(z_1,\dots,z_n)\longmapsto(I_1,\dots,I_{\lfloor n/2\rfloor})
\]
through \(\mathcal Z_n\):
\[
(\C^\times)^n
\longrightarrow
\mathcal Z_n
\longrightarrow
\C^{\lfloor n/2\rfloor}.
\]
Here the first map is given by
\[
(z_1,\dots,z_n)
\longmapsto
\mathbf Z(t):=Z_1(t)\cdots Z_n(t),
\]
with \(Z_i(t)\) defined by~\eqref{eq:zmatrix}, and the second map is
given by
\[
\mathbf Z(t)\longmapsto(I_1,\dots,I_{\lfloor n/2\rfloor}),
\]
with $I_j$ defined by
\[
\tr \mathbf Z(t)=2+\sum_{j=1}^{\lfloor n/2\rfloor} I_j t^j.
\]
\begin{lemma}\label{lem:trace-surjectivity}
The trace map
\[
\mathcal Z_n\longrightarrow \mathbb C^{\lfloor n/2\rfloor},
\qquad
\mathbf Z(t)\longmapsto (I_1,\dots,I_{\lfloor n/2\rfloor}),
\]
is surjective.
\end{lemma}

\begin{proof}
Let
\[
T(t)=2+\sum_{j=1}^{\lfloor n/2\rfloor} I_j t^j
\]
be arbitrary. First assume \(n=2m+1\). Choose
\[
A(t)=1,\qquad D(t)=T(t)-1.
\]
Then \(A(0)=D(0)=1\) and \(A(t)+D(t)=T(t)\). Moreover,
\(
A(t)D(t)-Q(t)
\)
has zero constant term and degree at most \(2m+1\). Therefore we may
factor it as
\[
A(t)D(t)-Q(t)=B(t)C(t),
\]
with
\[
B(0)=0,\qquad \deg B\le m+1,\qquad \deg C\le m.
\]
Now assume \(n=2m\). Choose
\[
A(t)=1+ct^m,\qquad D(t)=T(t)-1-ct^m.
\]
Then again \(A(0)=D(0)=1\) and \(A(t)+D(t)=T(t)\). Let
\(
\varepsilon:=\varepsilon_1 \cdots \eps_{2m}
\)
be the coefficient of \(t^{2m}\) in \(Q(t)\). Then the coefficient of
\(t^{2m}\) in \(A(t)D(t)-Q(t)\) is
\(
c(I_m-c)-\varepsilon.
\)
Choose \(c\) so that
\[
c(I_m-c)=\varepsilon.
\]
Then
\[
\deg\bigl(A(t)D(t)-Q(t)\bigr)\le 2m-1.
\]
Also, \(A(t)D(t)-Q(t)\) has zero constant term. Therefore we may factor
\[
A(t)D(t)-Q(t)=B(t)C(t),
\]
with
\[
B(0)=0,\qquad \deg B\le m,\qquad \deg C\le m-1.
\]
In both cases,
\[
\mathbf Z(t)=
\begin{pmatrix}
A(t)&B(t)\\
C(t)&D(t)
\end{pmatrix}
\]
belongs to \(\mathcal Z_n\), and by construction
\[
\tr \mathbf Z(t)=T(t).
\]
Hence the trace map is surjective.
\end{proof}
Define
\[
\mathcal Z_n^\circ
:=
\left\{
\mathbf Z(t)\in\mathcal Z_n
\;\middle|\;
\mathbf Z(1)\neq0,\ \mathbf Z(-1)\neq0
\right\}.
\]
\begin{lemma}\label{lem:right-factor}
Let \(\mathbf Z(t)\in\mathcal Z_n^\circ\). Assume that
\[
\ker \mathbf Z(\varepsilon_n)
=
\operatorname{span}
\begin{pmatrix}
1\\
-z_n
\end{pmatrix}
\]
for some \(z_n\in\mathbb C^\times\). Then
\[
\mathbf Z(t)=\widetilde{\mathbf Z}(t)Z_n(t),
\]
where \(Z_n(t)\) is given by \eqref{eq:zmatrix}, and
\(\widetilde{\mathbf Z}(t)\in\mathcal Z_{n-1}^\circ\).
\end{lemma}

\begin{proof}
Define
\begin{equation}\label{eq:tildez}
\widetilde{\mathbf Z}(t):=\mathbf Z(t)Z_n(t)^{-1}
=
\frac{\mathbf Z(t)\operatorname{adj}Z_n(t)}{1-\varepsilon_n t},
\end{equation}
where
\begin{equation}\label{eq:adjoint}
\operatorname{adj}Z_n(t)=
\begin{pmatrix}
1 & -\varepsilon_n t / z_n\\[1ex]
-z_n & 1
\end{pmatrix}.    
\end{equation}
We show that \(\widetilde{\mathbf Z}(t)\in\mathcal Z_{n-1}^\circ\).

\smallskip
\noindent\emph{Polynomiality.}
We have
\[
\operatorname{im}\operatorname{adj}Z_n(\varepsilon_n)
=
\operatorname{span}
\begin{pmatrix}
1\\
-z_n
\end{pmatrix}
=
\ker\mathbf Z(\varepsilon_n).
\]
Hence
\[
\mathbf Z(\varepsilon_n)\operatorname{adj}Z_n(\varepsilon_n)=0.
\]
Since \(\varepsilon_n=\pm1\), this implies that the matrix polynomial
\(
\mathbf Z(t)\operatorname{adj}Z_n(t)
\)
is divisible by \(1-\varepsilon_nt\). Therefore
\(\widetilde{\mathbf Z}(t)\) is a polynomial matrix.

\smallskip
\noindent\emph{Degree bounds.} We use the formula \eqref{eq:tildez} for \(\widetilde{\mathbf Z}(t)\). 
The degree bounds for \(\mathbf Z(t)\) are
\[
\begin{pmatrix}
\lfloor n/2\rfloor & \lfloor (n+1)/2\rfloor\\[0.4ex]
\lfloor (n-1)/2\rfloor & \lfloor n/2\rfloor
\end{pmatrix}.
\]
Multiplication by \eqref{eq:adjoint}
gives degree bounds
\[
\begin{pmatrix}
\lfloor (n+1)/2\rfloor & \lfloor n/2\rfloor+1\\[0.4ex]
\lfloor n/2\rfloor & \lfloor (n+1)/2\rfloor
\end{pmatrix}.
\]
Dividing by the linear factor \(1-\varepsilon_n t\) decreases each bound
by one. Therefore \(\widetilde{\mathbf Z}(t)\) satisfies the degree
bounds
\[
\begin{pmatrix}
\lfloor (n-1)/2\rfloor & \lfloor n/2\rfloor\\[0.4ex]
\lfloor (n-2)/2\rfloor & \lfloor (n-1)/2\rfloor
\end{pmatrix},
\]
which are precisely the bounds defining \(\mathcal Z_{n-1}\).

\smallskip
\noindent\emph{Determinant.}
We have
\[
\det\widetilde{\mathbf Z}(t)
=
\frac{\det\mathbf Z(t)}{\det Z_n(t)}
=
(1-\varepsilon_i t) \cdots (1 - \eps_{n-1}t).
\]

\smallskip
\noindent\emph{Unitriangularity.}
Since \(\mathbf Z(0)\) and \(Z_n(0)\) are both lower unitriangular, so is
\[
\widetilde{\mathbf Z}(0)=\mathbf Z(0)Z_n(0)^{-1}.
\]
This completes the proof that
\(\widetilde{\mathbf Z}(t)\in\mathcal Z_{n-1}\).

\smallskip
\noindent\emph{Nonvanishing at \(t=\pm1\).}
Since
\[
\mathbf Z(t)=\widetilde{\mathbf Z}(t)Z_n(t),
\]
the condition \(\mathbf Z(\pm 1)\neq0\) implies
\(
\widetilde{\mathbf Z}(\pm 1)\neq0
\). 
Therefore
\(
\widetilde{\mathbf Z}(t)\in\mathcal Z_{n-1}^\circ.
\) \qedhere
\end{proof}
\begin{lemma}\label{lem:openness}
The set of matrices \(\mathbf Z(t)\in\mathcal Z_n^\circ\) which admit a
factorization
\[
\mathbf Z(t)=Z_1(t)\cdots Z_n(t),
\]
with \(Z_i(t)\) as in \eqref{eq:zmatrix}, for some
\(z_1,\dots,z_n\in\mathbb C^\times\), is a Zariski open subset of
\(\mathcal Z_n^\circ\).
\end{lemma}

\begin{proof}
We argue by induction on \(n\). The case \(n=1\) is immediate.

For \(\mathbf Z(t)\in\mathcal Z_n^\circ\), we have
\(
\det \mathbf Z(\varepsilon_n)=0
\) and \(
\mathbf Z(\varepsilon_n)\neq0.
\)
Thus \(\mathbf Z(\varepsilon_n)\) has rank one and hence a
one-dimensional kernel. This defines a regular map
\[
\kappa_n\colon \mathcal Z_n^\circ\longrightarrow \mathbb{CP}^1,
\qquad
\mathbf Z(t)\longmapsto \ker \mathbf Z(\varepsilon_n).
\]
Let \(U_n\subset\mathcal Z_n^\circ\) be the open subset where
\(
\kappa_n(\mathbf Z)\notin\{0,\infty\}.
\)
Equivalently, \(\mathbf Z\in U_n\) if and only if there exists
\(z_n\in\mathbb C^\times\) such that
\[
\ker\mathbf Z(\varepsilon_n)
=
\operatorname{span}
\begin{pmatrix}
1\\
-z_n
\end{pmatrix}.
\]
In particular, this condition is satisfied if
\[
\mathbf Z(t)=Z_1(t)\cdots Z_n(t).
\]
Indeed, in that case we have
\[
\ker\mathbf Z(\varepsilon_n)\supseteq
\ker Z_n(\varepsilon_n)
=
\operatorname{span}
\begin{pmatrix}
1\\
-z_n
\end{pmatrix}.
\]
Moreover, since \(\mathbf Z(\varepsilon_n)\neq0\), the kernel of
\(\mathbf Z(\varepsilon_n)\) is one-dimensional, so this inclusion is
necessarily an equality. Thus every factorizable element belongs to \(U_n\). It therefore remains
to show that the set of factorizable elements is open in \(U_n\).

By Lemma~\ref{lem:right-factor}, for every \(\mathbf Z\in U_n\) we can
write
\[
\mathbf Z(t)=\widetilde{\mathbf Z}(t)Z_n(t),
\qquad
\widetilde{\mathbf Z}(t)\in\mathcal Z_{n-1}^\circ.
\]
Moreover, the map
\[
U_n\longrightarrow\mathcal Z_{n-1}^\circ,
\qquad
\mathbf Z\longmapsto\widetilde{\mathbf Z},
\]
is regular, because \(Z_n\) depends regularly on $\mathbf Z$ and
\[
\widetilde{\mathbf Z}(t)=\mathbf Z(t)Z_n(t)^{-1}.
\]
Let \(\mathcal F_n\subset\mathcal Z_n^\circ\) denote the factorizable
locus. Then
\[
\mathcal F_n
=
\{\mathbf Z\in U_n\mid \widetilde{\mathbf Z}\in\mathcal F_{n-1}\}.
\]
By the induction hypothesis, \(\mathcal F_{n-1}\) is open in
\(\mathcal Z_{n-1}^\circ\). Therefore \(\mathcal F_n\) is open in
\(U_n\), and hence open in \(\mathcal Z_n^\circ\).
\end{proof}
\begin{lemma}\label{lem:nonemptiness}
The factorizable locus in \(\mathcal Z_n^\circ\) is nonempty.
\end{lemma}

\begin{proof}
We argue by induction on \(n\). The case \(n=1\) is immediate. Assume
there exists a factorizable matrix
\(
\widetilde{\mathbf Z}(t)\in\mathcal Z_{n-1}^\circ.
\)
Choose \(z_n\in\mathbb C^\times\) and define
\[
\mathbf Z(t):=\widetilde{\mathbf Z}(t)Z_n(t),
\]
where \(Z_n(t)\) is given by \eqref{eq:zmatrix}. Then
\(\mathbf Z(t)\) is factorizable by construction and belongs to
\(\mathcal Z_n\). It remains to show that \(z_n\) can be chosen so that
\(
\mathbf Z( t)\neq0
\) for \(t=\pm1\).

Fix \(t=\pm1\). If \(t\neq\varepsilon_n\), then \(Z_n(t)\) is
invertible, and therefore \(\mathbf Z(t)\neq0\) because
\(\widetilde{\mathbf Z}(t)\neq0\).
If \(t=\varepsilon_n\) the condition
\(
\mathbf Z(\varepsilon_n)=0
\)
is equivalent to
\[
\ker\widetilde{\mathbf Z}(\varepsilon_n)\supset
\operatorname{im}Z_n(\varepsilon_n)
=
\operatorname{span}
\begin{pmatrix}
1\\
z_n
\end{pmatrix}.
\]
Since \(\widetilde{\mathbf Z}(\varepsilon_n)\neq0\), its kernel is
one-dimensional. Hence this condition excludes at most one value of
\(z_n\). We may therefore choose
\(z_n\in\mathbb C^\times\) avoiding this value. It follows that
\(
\mathbf Z(\varepsilon_n)\neq0,
\)
and therefore
\(
\mathbf Z(\pm1)\neq0.
\)
Thus
\(
\mathbf Z(t)\in\mathcal Z_n^\circ.
\)
\end{proof}

We are now ready to prove Proposition~\ref{prop:I-dominant}.

\begin{proof}[Proof of Proposition~\ref{prop:I-dominant}]
The map
\[
(\C^\times)^n
\longrightarrow
\C^{\lfloor n/2\rfloor},
\qquad
(z_1,\dots,z_n)
\longmapsto
(I_1,\dots,I_{\lfloor n/2\rfloor})
\]
is the composition of the following two maps:
\begin{align}
\mathrm{product}\colon\quad
(\C^\times)^n
&\longrightarrow
\mathcal Z_n,
&
(z_1,\dots,z_n)
&\longmapsto
Z_1(t)\cdots Z_n(t), \label{eq:product-map1}\\
\mathrm{trace}\colon\quad
\mathcal Z_n
&\longrightarrow
\C^{\lfloor n/2\rfloor},
&
\mathbf Z(t)
&\longmapsto
(I_1,\dots,I_{\lfloor n/2\rfloor}). \label{eq:trace-map1}
\end{align}

By Lemma~\ref{lem:trace-surjectivity}, the trace
map~\eqref{eq:trace-map1} is surjective, and hence dominant. By
Lemmas~\ref{lem:openness} and~\ref{lem:nonemptiness}, the image of the
product map~\eqref{eq:product-map1} contains a nonempty Zariski open
subset of \(\mathcal Z_n^\circ\), and hence of \(\mathcal Z_n\).
Therefore the product map is dominant. Since the composition of dominant
morphisms is dominant, the map
\[
(\C^\times)^n
\longrightarrow
\C^{\lfloor n/2\rfloor},
\qquad
(z_1,\dots,z_n)
\longmapsto
(I_1,\dots,I_{\lfloor n/2\rfloor})
\]
is dominant.
\end{proof}

\subsection{Poisson brackets}

In this section, we introduce a Poisson bracket on
\(
\mathcal P_n/\Aff(\mathbb C)
\)
that is preserved by the solution map \(\mathcal S\), and prove that the
first integrals \(I_j\) pairwise Poisson commute. 

Since the solution map admits a zero-curvature representation, it can be
realized as a refactorization-type transformation. Poisson--Lie theory
therefore provides a natural framework for constructing the associated
Poisson structure. In particular, the bracket introduced below could be
derived from the standard trigonometric \(r\)-matrix. Instead, we first
define it \emph{ad hoc} and verify that it is preserved by every
folding, and hence by the solution map. Only afterwards do we identify
it with the Poisson structure arising from the trigonometric
\(r\)-matrix and use this identification to prove that the first
integrals~\(
I_j
\)
Poisson commute.


Recall that the quotient
\(
\mathcal P_n/\Aff(\mathbb C)
\)
is identified with the hypersurface
\begin{equation}\label{eq:prod-y}
y_1\cdots y_n=1.
\end{equation}
in $(\C^\times)^n$. The identification sends the equivalence class of a polygon
\((p_i)\in\mathcal P_n\) modulo affine transformations to the point
with coordinates
\[
y_i:=\frac{z_i}{z_{i-1}},
\qquad
z_i:=p_{i+1}-p_i.
\]
We equip the torus $(\C^\times)^n$ with the standard cyclic log-canonical Poisson bracket
\begin{equation}\label{eq:logcanonical}
\{y_i,y_j\}=
\begin{cases}
y_i y_j, & j=i+1 \!\!\!\pmod n,\\
-y_i y_j, & i=j+1 \!\!\!\pmod n,\\
0, & \text{otherwise}.
\end{cases}
\end{equation}
The function
\(
y_1\cdots y_n
\)
is a Casimir of this bracket. Therefore, each of its level sets, and in
particular the hypersurface
\eqref{eq:prod-y}
is a Poisson subvariety of \((\C^\times)^n\). Under the identification
\[
\mathcal P_n/\Aff(\mathbb C)
\cong
\{\,y_1\cdots y_n=1\,\},
\]
the Poisson bracket therefore descends to
\(\mathcal P_n/\Aff(\mathbb C)\).

\begin{proposition}\label{prop:invPB}
The solution map \(\mathcal S\) preserves the Poisson bracket
\eqref{eq:logcanonical}.
\end{proposition}

\begin{proof}
Since the solution map \(\mathcal S\) is a composition of folding maps
\(F_j\), it suffices to show that each \(F_j\) preserves the bracket
\eqref{eq:logcanonical}. In the coordinates \(y_i=z_i/z_{i-1}\), the folding \(F_j\) is given by
\begin{equation}\label{eq:yfolding}
y_{j-1}'=-y_{j-1}\frac{y_j+1}{y_j-1},
\qquad
y_j'=-y_j,
\qquad
y_{j+1}'=y_{j+1}\frac{y_j-1}{y_j+1},
\end{equation}
while all other coordinates remain unchanged. A direct computation shows
that the transformed coordinates satisfy the same bracket relations
\eqref{eq:logcanonical}. Hence each \(F_j\) is Poisson, and therefore so
is \(\mathcal S\).
\end{proof}

\begin{remark}
The transformation \eqref{eq:yfolding} is closely related to \(Y\)-mutations in cluster algebra theory and becomes the standard \(Y\)-mutation after an appropriate change of variables. The preservation of the log-canonical bracket \eqref{eq:logcanonical} may therefore be viewed as a manifestation of the well-known compatibility between \(Y\)-mutations and log-canonical Poisson structures.
\end{remark}

\begin{remark}
The Poisson bracket \eqref{eq:logcanonical} is local: the Poisson bracket
of two sufficiently distant coordinates \(y_i\) vanishes. This contrasts with the geometric structure previously associated with
discrete conformal maps, which arises from their Lagrangian formulation
and the three-leg form of the cross-ratio equation
\cite{bobenko2002integrable,adler2003classification}. That structure is,
in general, presymplectic, and even when it can be inverted, the
resulting Poisson bracket is non-local. Moreover, it is not known whether
the first integrals arising from the zero-curvature representation
Poisson commute with respect to that bracket. The local Poisson bracket introduced here is therefore genuinely
different from the one coming from the Lagrangian theory and appears to
be better suited for the study of integrability.
\end{remark}





We now show that the first integrals
\(
I_1,\dots,I_{\lfloor n/2\rfloor}
\)
Poisson commute with respect to the bracket
\eqref{eq:logcanonical}. Recall that these integrals are the
coefficients of the polynomial
\[
\tr\bigl(Y_1(t)\cdots Y_n(t)\bigr),
\]
where the matrices \(Y_i(t)\) are given by~\eqref{eq:ymatrix}.
\begin{proposition}\label{prop:integrals-commute}
The functions \(I_j\) Poisson commute with respect to the bracket \eqref{eq:logcanonical}.
\end{proposition}

\begin{proof}
We prove a slightly more general statement. Namely, we show that the coefficients of
\(
\tr\left(Y_1'\cdots Y_n'\right)
\)
where
\begin{equation}\label{eq:Yprime}
Y_i'=
\begin{pmatrix}
1&\eps_{i-1} y_i \lambda^{-1}\\
1&y_i
\end{pmatrix}, 
\end{equation}
Poisson commute with respect to the bracket \eqref{eq:logcanonical}, without imposing the condition
\eqref{eq:prod-y}.
The proposition then follows by restricting to the Poisson subvariety \eqref{eq:prod-y} and setting \(t=\lambda^{-1}\).

The proof is based on an \(r\)-matrix realization of the above Poisson bracket, constructed as follows. Consider the associative algebra \(\Mat_2(\C[\lambda,\lambda^{-1}])\) of Laurent polynomials with coefficients in \(\Mat_2(\C)\). It carries a multiplicative Sklyanin--Adler--Gelfand--Dickey type Poisson bracket defined by the trigonometric \(r\)-matrix. Details on this bracket can be found, for instance, in \cite{izosimov2022pentagram}, where it is described in the language of difference operators. The isomorphism between \(k\)-periodic difference operators and \(\Mat_k(\C[\lambda,\lambda^{-1}])\) is explained in Remark~3.1 of loc.~cit.; here we use this correspondence for \(k=2\).

In particular, this Poisson bracket restricts to the subvariety
\[
\mathcal X
=
\left\{
\begin{pmatrix}
a&b\lambda^{-1}\\
c&d
\end{pmatrix}
\;\middle|\;
a,b,c,d\in\C^\times
\right\}
\subset \Mat_2(\C[\lambda,\lambda^{-1}]),
\]
which corresponds to degree-one difference operators. In coordinates, the restricted bracket is given by
\begin{align*}
\{b,a\}&=\frac12ab,
&
\{a,c\}&=\frac12ac,
&
\{c,d\}&=\frac12cd,
&
\{d,b\}&=\frac12bd,
\\
&&\{a,d\}&=0,
&
\{b,c\}&=0.
\end{align*}

Consider the product \(\mathcal X^n\), equipped with the product Poisson structure. The coordinates \(a,b,c,d\) on \(\mathcal X\) induce coordinates \(a_i,b_i,c_i,d_i\) on the \(i\)-th factor. Since \(a_id_i\) and \(b_ic_i\) are Casimir functions on each factor, the equations
\[
\frac{b_i c_i}{a_i d_i}=\eps_{i-1},
\qquad i=1,\dots,n, 
\]
define a Poisson subvariety
\[
\mathcal X^n_\eps
=
\left\{
(A_1,\dots,A_n)\in \mathcal X^n
\;\middle|\;
\frac{b_i c_i}{a_i d_i}=\eps_{i-1}
\text{ for all } i
\right\}. 
\]
The multiplication map
\[
\mathcal X^n_\eps
\longrightarrow
\Mat_2(\C[\lambda,\lambda^{-1}]),
\qquad
(A_1,\dots,A_n)\longmapsto A_1\cdots A_n,
\]
is Poisson, by multiplicativity of the bracket on
\(\Mat_2(\C[\lambda,\lambda^{-1}])\). Furthermore, since the bracket on \(\Mat_2(\C[\lambda,\lambda^{-1}])\) is given by an \(r\)-matrix, spectral invariants are in involution. In particular, the coefficients of \(\tr M(\lambda)\), viewed as functions on \(\Mat_2(\C[\lambda,\lambda^{-1}])\), Poisson commute. Therefore, their pullbacks to \(\mathcal X^n_\eps\) under the multiplication map also Poisson commute. These pullbacks are precisely the coefficients of
\(
\tr(A_1\cdots A_n).
\)

Next consider the gauge action on \(\mathcal X^n_\eps\) of the group \((\C^\times)^{2n}\) of constant invertible diagonal matrices
\(
D_i=\operatorname{diag}(p_i,q_i),
\)
given by
\[
(A_1,\dots,A_n)
\longmapsto
(D_1^{-1}A_1D_2,\,
D_2^{-1}A_2D_3,\,
\dots,\,
D_n^{-1}A_nD_1).
\]
This action is Poisson, where \((\C^\times)^{2n}\) is equipped with the zero Poisson structure. Indeed, this is a standard property of the trigonometric \(r\)-matrix bracket; it can also be verified directly from the induced action on the coordinates \(a_i,b_i,c_i,d_i\). Therefore, the quotient \(\mathcal X^n_\eps/(\C^\times)^{2n}\) inherits a Poisson structure.

The trace coefficients descend to the quotient. Indeed, under the gauge action,
 the product $A_1\cdots A_n$ is transformed by conjugation. Hence each coefficient of
\(
\tr(A_1\cdots A_n)
\)
is invariant under the gauge action. Since these trace coefficients are in involution on \(\mathcal X^n_\eps\), their descents to the quotient \(\mathcal X^n_\eps/(\C^\times)^{2n}\) are likewise in involution.

The quotient \(\mathcal X^n_\eps/(\C^\times)^{2n}\) is isomorphic to
\((\C^\times)^{n+1}\). An explicit isomorphism is given by the gauge-invariant functions
\[
y_i:=\frac{c_{i+1}d_i}{c_i a_{i+1}},
\qquad
\rho:=a_1a_2\cdots a_n,
\]
where the indices are understood cyclically. A direct computation shows that the induced Poisson bracket in these coordinates is given by \eqref{eq:logcanonical}, together with
\begin{equation}\label{eq:rhocas}
    \{\rho,y_i\}=0,
\end{equation}
so that \(\rho\) is a Casimir.

We have therefore shown that the descended trace coefficients Poisson commute with respect to this Poisson structure. It remains to identify these descended functions explicitly. To that end, consider the map
\[
s:\mathcal X^n_\eps/(\C^\times)^{2n}\longrightarrow \mathcal X^n_\eps,
\qquad
(y_1,\dots,y_n,\rho)\longmapsto (Y_1',\dots,Y_n'),
\]
where \(Y_i'\) is given by \eqref{eq:Yprime} for \(i=1,\dots,n-1\), and
\[
Y_n'=
\begin{pmatrix}
\rho&\rho\eps_{n-1} y_n\lambda^{-1}\\
1&y_n
\end{pmatrix}.
\]
This map is a section of the quotient map
\[
\pi:\mathcal X^n_\eps\to \mathcal X^n_\eps/(\C^\times)^{2n}.
\]
We may therefore identify the quotient with the image of this section. Under this identification, the descended trace coefficients become precisely the coefficients of
\(
\tr(Y_1'\cdots Y_n').
\)
Hence these coefficients Poisson commute with respect to the bracket given by \eqref{eq:logcanonical} and \eqref{eq:rhocas}.
Since \(\rho\) is a Casimir, restricting to the level set \(\rho=1\) preserves Poisson commutativity. On this level, the matrix \(Y_n'\) is also given by \eqref{eq:Yprime}. Therefore the coefficients of
\(
\tr(Y_1'\cdots Y_n'),
\)
where all \(Y_i'\) are given by \eqref{eq:Yprime}, Poisson commute.
\end{proof}

\subsection{Complete integrability of the Cauchy problem (Proof of Theorem~\ref{thm1})}

\begin{proof}[Proof of Theorem~\ref{thm1}]
By Proposition~\ref{prop:invPB}, the birational map
\[
\mathcal S\colon
\mathcal P_n/\Aff(\C)
\dashrightarrow
\mathcal P_n/\Aff(\C)
\]
preserves the Poisson structure \eqref{eq:logcanonical}. By
Proposition~\ref{prop:integrals-commute}, the functions
\(
I_1,\dots,I_{\lfloor n/2\rfloor}
\)
pairwise Poisson commute. They are algebraically independent and are
preserved by \(\mathcal S\). It remains to show that they form a
complete set of commuting first integrals. Equivalently, it remains to
verify that their number is
\[
\frac12\operatorname{rank}(\{\cdot,\cdot\})
+
\operatorname{corank}(\{\cdot,\cdot\}).
\]
If \(n\) is odd, the Poisson structure is symplectic. Hence
\[
\operatorname{rank}(\{\cdot,\cdot\})=n-1,
\qquad
\operatorname{corank}(\{\cdot,\cdot\})=0.
\]
If \(n\) is even, the Poisson structure has the Casimir
\(
y_1y_3\cdots y_{n-1},
\)
and therefore
\[
\operatorname{rank}(\{\cdot,\cdot\})=n-2,
\qquad
\operatorname{corank}(\{\cdot,\cdot\})=1.
\]
In both cases,
\[
\frac12\operatorname{rank}(\{\cdot,\cdot\})
+
\operatorname{corank}(\{\cdot,\cdot\})
=
\lfloor n/2\rfloor,
\]
which is precisely the number of first integrals. Thus
\(
I_1,\dots,I_{\lfloor n/2\rfloor}
\)
form a complete set of commuting first integrals.
\end{proof}

\section{Orthogonal square grid circle patterns}

\subsection{Definition}

We begin with the definition of an orthogonal square grid circle pattern. The definition is chosen so that every orthogonal square grid circle pattern is automatically a discrete conformal map.

 Recall that a quadrilateral \(ABCD\) is called a \emph{right kite} if 
\[
AB=AD,\qquad BC=CD, \qquad
\angle B=\angle D=\frac{\pi}{2},
\]
see Figure \ref{fig:kite}. Informally, an orthogonal square grid circle pattern is a map
\(
f:\mathbb Z^2\to\mathbb C
\)
such that the image of every elementary lattice quadrilateral is a right kite. Every right kite is a harmonic quadrilateral. Hence every orthogonal square grid circle pattern is a discrete conformal map.

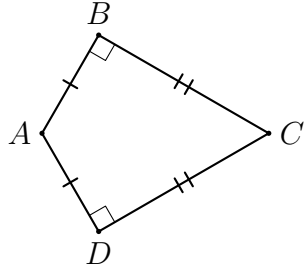
\begin{figure}[!htb]
    \centering
    \begin{tikzpicture}[
  scale=0.75,
  line cap=round,
  line join=round
]

\pgfmathsetmacro{\r}{2}
\pgfmathsetmacro{\h}{sqrt(3)*\r/2}

\coordinate (A) at (0,0);
\coordinate (C) at ({2*\r},0);
\coordinate (B) at ({\r/2},{\h});
\coordinate (D) at ({\r/2},{-\h});

\tikzset{
  one tick/.style={
    postaction={
      decorate,
      decoration={
        markings,
        mark=at position 0.5 with {
          \draw (0,-3pt) -- (0,3pt);
        }
      }
    }
  },
  two ticks/.style={
    postaction={
      decorate,
      decoration={
        markings,
        mark=at position 0.47 with {
          \draw (0,-3pt) -- (0,3pt);
        },
        mark=at position 0.53 with {
          \draw (0,-3pt) -- (0,3pt);
        }
      }
    }
  }
}

\draw[ thick,one tick] (A) -- (B);
\draw[ thick,two ticks] (B) -- (C);
\draw[ thick,two ticks] (C) -- (D);
\draw[ thick,one tick] (D) -- (A);


\pic[draw,angle radius=7pt] {right angle = A--B--C};
\pic[draw,angle radius=7pt] {right angle = C--D--A};

\fill (A) circle (1.3pt) node[left] {$A$};
\fill (B) circle (1.3pt) node[above] {$B$};
\fill (C) circle (1.3pt) node[right] {$C$};
\fill (D) circle (1.3pt) node[below] {$D$};

\end{tikzpicture}
    \caption{A right kite.}
    \label{fig:kite}
\end{figure}

To formulate the definition precisely, let
\[
\mathbb Z^2_0=\{(m,n)\in\mathbb Z^2:m+n\equiv0\pmod 2\},
\quad
\mathbb Z^2_1=\{(m,n)\in\mathbb Z^2:m+n\equiv1\pmod 2\}.
\]

\begin{definition}\label{def:osgcp}
An \emph{orthogonal square grid circle pattern} is a map
\[
f\colon\mathbb Z^2\to\mathbb C,
\qquad
v\longmapsto f_v,
\]
such that the vertices of every elementary lattice square are mapped to
four distinct points and the following conditions hold.

\begin{enumerate}
\item For every \(v\in\mathbb Z^2_0\), if
\(
w_1,w_2,w_3,w_4\in\mathbb Z^2_1
\)
are the four nearest neighbors of \(v\), then
\[
|f_v-f_{w_1}|
=
|f_v-f_{w_2}|
=
|f_v-f_{w_3}|
=
|f_v-f_{w_4}|.
\]

\item For every \(w\in\mathbb Z^2_1\), let
\(
v_1,v_2,v_3,v_4\in\mathbb Z^2_0
\)
be the four nearest neighbors of \(w\), listed in cyclic order around
\(w\). Then
\[
(f_{v_i}-f_w)
\perp
(f_{v_{i+1}}-f_w),
\qquad
i=1,\dots,4,
\]
where the indices are understood modulo \(4\).
\end{enumerate}
\end{definition}

The geometric interpretation of the definition is as follows. For every
\(v\in\mathbb Z^2_0\), condition~(1) implies that there exists a circle
centered at \(f_v\) and passing through the images of the four neighbors
of \(v\); denote it by \(C_v\). The circles \(C_v\) and
\(C_{v+(2,0)}\) pass through the common point \(f_{v+(1,0)}\). Since
the segments \(f_{v+(1,0)}-f_v\) and
\(f_{v+(1,0)}-f_{v+(2,0)}\) are collinear by condition~(2), these
circles are tangent. The same holds for \(C_v\) and \(C_{v+(0,2)}\).
Therefore
\[
\{C_v\mid v\in2\mathbb Z^2\}
\quad\text{and}\quad
\{C_v\mid v\in(1,1)+2\mathbb Z^2\}
\]
are circle packings.
Now let \(v,w\in\mathbb Z^2_0\) satisfy
\[
w-v\in\{(\pm1,\pm1)\}.
\]
Then \(C_v\) and \(C_w\) pass through each of the points
\(f_u\), where \(u\) is a common neighbor of \(v\) and \(w\). At either
intersection point, the corresponding radii are orthogonal by
condition~(2). Hence \(C_v\) and \(C_w\) intersect orthogonally.

Thus an orthogonal square grid circle pattern consists of two mutually
orthogonal circle packings; see Figure~\ref{fig:kitescircles}. This
recovers the geometric description given in the introduction.

\begin{figure}[H]
  \centering

\begin{minipage}[c]{0.3\textwidth}
    \centering
    \includegraphics[height=0.7\linewidth]{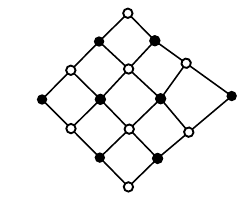}
\end{minipage}
\begin{minipage}[c]{0.20\textwidth}
    \centering
    \begin{tikzpicture}
\draw[-Latex, dashed, very thick] (0,0) -- (1,0);
\end{tikzpicture}
\end{minipage}
\begin{minipage}[c]{0.3\textwidth}
    \centering
    \includegraphics[height=0.7\linewidth]{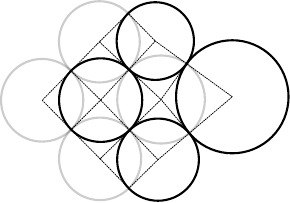}
\end{minipage}
    \caption{Equivalence of two definitions of orthogonal square grid circle patterns. From Definition \ref{def:osgcp} to two mutually orthogonal circle packings. Vertices of \(\mathbb Z^2_0\) are shown in black and vertices of \(\mathbb Z^2_1\) in white. }
    \label{fig:kitescircles}
\end{figure}
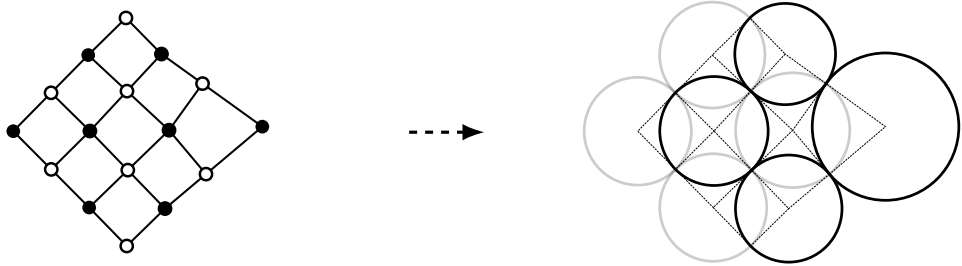

\begin{proposition}
Every orthogonal square grid circle pattern is a discrete conformal map.
\end{proposition}

\begin{proof}
Let \(p_1,p_2,p_3,p_4\) be the images of the vertices of an elementary
lattice square, listed in cyclic order. We use cyclic indices modulo \(4\)
and set
\[
z_i=p_{i+1}-p_{i},
\qquad i=1,\dots,4.
\]
Since the vertices of every elementary lattice square are mapped to
distinct points, all \(z_i\) are nonzero. Define
\[
y_i=\frac{z_{i}}{z_{i-1}},
\qquad i=1,\dots,4.
\]

By the definition of an orthogonal square grid circle pattern, two
opposite elements of the collection \(y_1,y_2,y_3,y_4\) lie on the
imaginary axis, while the other two lie on the unit circle. Without loss
of generality, we may assume that
\[
y_1,y_3\in i\mathbb R,
\qquad
|y_2|=|y_4|=1.
\]
Since
\(
y_1y_2y_3y_4
=
1,
\)
we obtain
\[
y_1y_3=(y_2y_4)^{-1}.
\]
The left-hand side is real, being the product of two purely imaginary
numbers, while the right-hand side lies on the unit circle. Hence
\[
y_1y_3=\pm1.
\]
On the other hand,
\[
[p_1,p_2,p_3,p_4]
=
\frac{(p_1-p_2)(p_3-p_4)}
     {(p_2-p_3)(p_4-p_1)}
=
\frac{z_2z_4}{z_3z_1}
=
y_1y_3.
\]
Therefore
\[
[p_1,p_2,p_3,p_4]=\pm1.
\]
If
\(
[p_1,p_2,p_3,p_4]=1,
\)
then
either \(p_1=p_3\) or \(p_2=p_4\), contradicting the assumption
that the four vertices of an elementary lattice square are mapped to
distinct points. Consequently,
\[
[p_1,p_2,p_3,p_4]=-1.
\]

Thus every elementary quadrilateral is harmonic, and therefore \(f\) is
a discrete conformal map.
\end{proof}

\subsection{Initial value problem}

In this section, we formulate the initial value problem for orthogonal
square grid circle patterns. We show that the corresponding solution map
is obtained by restricting the solution map for discrete conformal maps
to a suitable invariant subset of polygons, called \emph{right
multikites}.
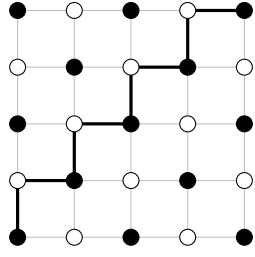
\begin{figure}[!htb]

\centering
\begin{tikzpicture}[scale = 1.5]
\draw[step=0.5cm,color=gray, opacity = 0.5] (0,0) grid (2,2);

\draw[very thick, shorten <=2.5pt, shorten >=2.5pt]
  (0,0) -- (0,0.5) -- (0.5,0.5) -- (0.5,1) -- (1,1)
  -- (1,1.5) -- (1.5,1.5) -- (1.5,2) -- (2,2);

\foreach \i in {0,...,4} {
  \foreach \j in {0,...,4} {
    \pgfmathtruncatemacro{\parity}{mod(\i+\j,2)}
    \ifnum\parity=0
      \filldraw[fill=black, draw=black] (0.5*\i,0.5*\j) circle (2pt);
    \else
      \filldraw[fill=white, draw=black] (0.5*\i,0.5*\j) circle (2pt);
    \fi
  }
}
\end{tikzpicture}
\caption{A staircase. Vertices of \(\mathbb Z^2_0\) are shown in black and vertices of \(\mathbb Z^2_1\) in white.}
\label{fig:staircase}
\end{figure}

Since every orthogonal square grid circle pattern is a discrete
conformal map, it is uniquely determined by its values on any zigzag.
In this paper we only consider zigzags of staircase type. We assume
that the staircase starts at a vertex
\(
v_1\in\mathbb Z^2_0
\)
and initially proceeds upward, as shown in
Figure~\ref{fig:staircase}. 

Let \(\zeta\) be a staircase zigzag,
and set
\[
p_i:=f_{\zeta_i},
\qquad
z_i:=p_{i+1}-p_i.
\]
If \(f\) is an orthogonal square grid circle pattern, then
\begin{equation}\label{eq:staircase}
z_{2k-1}\perp z_{2k},
\qquad
|z_{2k}|=|z_{2k+1}|,
\end{equation}
for every \(k\).

\begin{definition}
A polygon \((p_i)\) whose edge vectors satisfy
\eqref{eq:staircase} is called a \emph{right multikite}; see
Figure~\ref{fig:multikite}.
\end{definition}

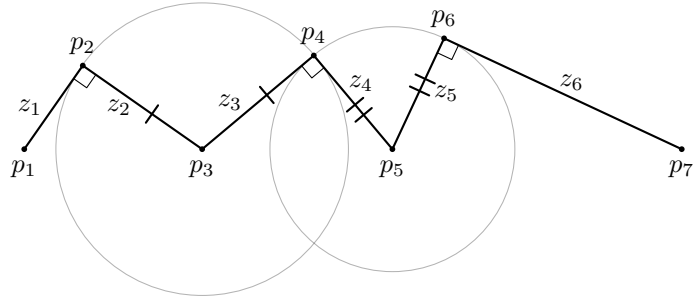
\begin{figure}[!htb]
\centering
\usetikzlibrary{angles,quotes,decorations.markings}

\begin{tikzpicture}[
  scale=0.9,
  line cap=round,
  line join=round
]

\tikzset{
  tick one/.style={
    postaction={decorate,
      decoration={markings,
        mark=at position 0.58 with {
          \draw (0pt,-4pt) -- (0pt,4pt);
        }
      }
    }
  },
  tick two/.style={
    postaction={decorate,
      decoration={markings,
        mark=at position 0.58 with {
          \draw (-2.5pt,-4pt) -- (-2.5pt,4pt);
          \draw ( 2.5pt,-4pt) -- ( 2.5pt,4pt);
        }
      }
    }
  },
  tick three/.style={
    postaction={decorate,
      decoration={markings,
        mark=at position 0.58 with {
          \draw (-4pt,-4pt) -- (-4pt,4pt);
          \draw ( 0pt,-4pt) -- ( 0pt,4pt);
          \draw ( 4pt,-4pt) -- ( 4pt,4pt);
        }
      }
    }
  },
}

\pgfmathsetmacro{\thetaA}{55}
\pgfmathsetmacro{\thetaB}{40}
\pgfmathsetmacro{\thetaC}{65}

\pgfmathsetmacro{\rA}{1.5}
\pgfmathsetmacro{\rB}{\rA*tan(\thetaA)}
\pgfmathsetmacro{\rC}{\rB*tan(\thetaB)}

\pgfmathsetmacro{\LA}{\rA/cos(\thetaA)}
\pgfmathsetmacro{\LB}{\rB/cos(\thetaB)}
\pgfmathsetmacro{\LC}{\rC/cos(\thetaC)}

\pgfmathsetmacro{\xA}{0}
\pgfmathsetmacro{\xB}{\LA}
\pgfmathsetmacro{\xC}{\LA+\LB}

\coordinate (p1) at (\xA,0);
\coordinate (p3) at ({\xA+\LA},0);
\coordinate (p5) at ({\xB+\LB},0);
\coordinate (p7) at ({\xC+\LC},0);

\coordinate (p2) at ({\xA+\rA*cos(\thetaA)},{\rA*sin(\thetaA)});
\coordinate (q1) at ({\xA+\rA*cos(\thetaA)},{-\rA*sin(\thetaA)});

\coordinate (p4) at ({\xB+\rB*cos(\thetaB)},{\rB*sin(\thetaB)});
\coordinate (q2) at ({\xB+\rB*cos(\thetaB)},{-\rB*sin(\thetaB)});

\coordinate (p6) at ({\xC+\rC*cos(\thetaC)},{\rC*sin(\thetaC)});
\coordinate (q3) at ({\xC+\rC*cos(\thetaC)},{-\rC*sin(\thetaC)});

\draw[thick]  (p1) -- node[pos=.5,left] {\footnotesize$z_1$} (p2);
\draw[thick,tick one]  (p2) -- node[pos=.5,left] {\footnotesize$z_2$} (p3);

\draw[thick,tick one]   (p3) -- node[pos=.5,left] {\footnotesize$z_3\,$} (p4);
\draw[thick,tick two] (p4) -- node[pos=.3,right] {\footnotesize$z_4$} (p5);

\draw[thick,tick two] (p5) -- node[pos=.5,right] {\footnotesize$\,z_5$} (p6);
\draw[thick]  (p6) -- node[pos=.4,right] {\footnotesize$\,\,z_6$} (p7);


\pic[draw,angle radius=6pt] {right angle = p1--p2--p3};

\pic[draw,angle radius=6pt] {right angle = p3--p4--p5};

\pic[draw,angle radius=6pt] {right angle = p5--p6--p7};

\fill (p1) circle (1.2pt) node[below] {\footnotesize$p_1$};
\fill (p2) circle (1.2pt) node[above] {\footnotesize$p_2$};
\fill (p3) circle (1.2pt) node[below] {\footnotesize$p_3$};
\fill (p4) circle (1.2pt) node[above] {\footnotesize$p_4$};
\fill (p5) circle (1.2pt) node[below] {\footnotesize$p_5$};
\fill (p6) circle (1.2pt) node[above] {\footnotesize$p_6$};
\fill (p7) circle (1.2pt) node[below] {\footnotesize$p_7$};

\draw[
    opacity=0.3
] (p3) circle[radius=2.15];
\draw[
    opacity=0.3
] (p5) circle[radius=1.8];


\end{tikzpicture}
\caption{A right multikite.}
\label{fig:multikite}
\end{figure}

\begin{remark}
Right multikites admit a simple description in terms of circle
configurations. For each odd vertex \(p_i\), let \(C_i\) denote the
circle centered at \(p_i\) and passing through the adjacent vertices
\(p_{i-1}\) and \(p_{i+1}\), as illustrated in
Figure~\ref{fig:multikite}. The defining right-angle condition is then
equivalent to the requirement that each pair of consecutive circles
\(C_i\) and \(C_{i+1}\) intersect orthogonally.

Conversely, let \(C_i\) be a sequence of circles satisfying
\(C_i\perp C_{i+1}\) for every \(i\). Taking the centers of the circles
as the odd vertices and choosing, for each consecutive pair
\(C_i,C_{i+1}\), one of their two intersection points as the
intermediate even vertex reconstructs a right multikite. Thus right
multikites are in one-to-one correspondence with sequences of circles
\(C_i\) satisfying \(C_i\perp C_{i+1}\), together with a choice of one
intersection point for each consecutive pair.

\end{remark}

The following proposition shows that condition
\eqref{eq:staircase} characterizes orthogonal square grid circle
patterns among discrete conformal maps.

\begin{proposition}
Let \(f:\mathbb Z^2\to\mathbb C\) be a discrete conformal map. If the
initial data on a staircase zigzag form a right multikite, then \(f\) is
an orthogonal square grid circle pattern.
\end{proposition}
\begin{proof}
We use the following elementary observation: a harmonic quadrilateral
with a right angle is a right kite, and so is a harmonic quadrilateral
with two equal adjacent sides.

Assume that the image of the solid staircase in
Figure~\ref{fig:staircase2} is a right multikite. Since \(f\) is a
discrete conformal map, the image of every shaded cell is a harmonic
quadrilateral. Moreover, every such quadrilateral has a right angle by
the definition of a right multikite. Hence the image of each shaded cell
is a right kite. It follows that the image of the dotted staircase is
again a right multikite.

\begin{figure}[!htb]
\centering
\begin{tikzpicture}[scale=1.5]
\draw[step=0.5cm, color=gray, opacity=0.5] (0,0) grid (2,2);

\foreach \k in {0,1,2,3} {
  \fill[gray, opacity=0.18]
    ({0.5*\k},{0.5*\k})
    -- ({0.5*\k},{0.5*\k+0.5})
    -- ({0.5*\k+0.5},{0.5*\k+0.5})
    -- ({0.5*\k+0.5},{0.5*\k})
    -- cycle;
}

\draw[very thick, shorten <=2.5pt, shorten >=2.5pt]
  (0,0) -- (0,0.5) -- (0.5,0.5) -- (0.5,1) -- (1,1)
  -- (1,1.5) -- (1.5,1.5) -- (1.5,2) -- (2,2);

\draw[very thick, dotted, shorten <=2.5pt, shorten >=2.5pt]
  (0,0) -- (0.5,0) -- (0.5,0.5) -- (1,0.5)
  -- (1,1) -- (1.5,1) -- (1.5,1.5) -- (2,1.5) -- (2,2);

\foreach \i in {0,...,4} {
  \foreach \j in {0,...,4} {
    \pgfmathtruncatemacro{\parity}{mod(\i+\j,2)}
    \ifnum\parity=0
      \filldraw[fill=black, draw=black] (0.5*\i,0.5*\j) circle (2pt);
    \else
      \filldraw[fill=white, draw=black] (0.5*\i,0.5*\j) circle (2pt);
    \fi
  }
}
\end{tikzpicture}
\caption{Propagation of intial data to next staircase.}
\label{fig:staircase2}
\end{figure}
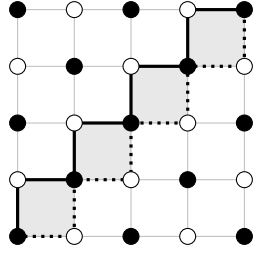

Repeating the same argument propagates the right multikite condition from
one staircase to the next. At each step, the cells between consecutive
staircases are mapped to harmonic quadrilaterals which either have a
right angle or have two equal adjacent sides, and hence are right
kites. Since \(\mathbb Z^2\) is covered by a sequence of pairwise adjacent staircases, every
cell of \(\mathbb Z^2\) is mapped to a right kite. Therefore \(f\) is an
orthogonal square grid circle pattern.
\end{proof}

As in the case of general discrete conformal maps, we shall be interested
in orthogonal square grid circle patterns that are quasi-periodic. We take
the period vector to be
\(
{\per}=(m,m),
\)
which is compatible with the staircase-shaped initial data considered
above. Thus we study orthogonal square grid circle patterns satisfying
\[
f_{i+m,j+m}=f_{i,j}+\Delta
\]
for some \(\Delta\in\C\).

For \(n=2m\), let
\(
\mathcal K_n\subset\mathcal P_n
\)
denote the set of quasi-periodic \(n\)-gons that are right multikites. By the preceding proposition, \(\mathcal K_{n}\)
coincides with the set of staircase initial data of quasi-periodic
orthogonal square grid circle patterns. In particular,
\(\mathcal K_{n}\) is invariant under the solution map~\(\mathcal S\).

Therefore, iterating the solution map on \(\mathcal K_n\) is equivalent
to solving the initial value problem for quasi-periodic orthogonal square
grid circle patterns. Passing to the quotient by affine
transformations, we obtain the induced birational map
\begin{equation}\label{eq:rsmap}
\mathcal S|_{\mathcal K_n/\Aff(\C)}\colon
\mathcal K_n/\Aff(\C)
\dashrightarrow
\mathcal K_n/\Aff(\C).
\end{equation}
In the following sections, we show that this quotient map is completely
integrable. To this end, we show that the geometric structures on
\(
\mathcal P_n/\Aff(\mathbb C)
\)
introduced for quasi-periodic discrete conformal maps restrict naturally
to
\(
\mathcal K_n/\Aff(\mathbb C)
\).


\subsection{The multikite space as a real part of a symplectic leaf}
In this section, we show that
\(
\mathcal K_n/\Aff(\mathbb C)
\)
is the fixed-point set of a suitable anti-regular involution on a
symplectic leaf of
\(
\mathcal P_n/\Aff(\mathbb C)
\),
and hence the real locus of that leaf. We further show that the
symplectic form on the leaf induces a symplectic structure on
\(
\mathcal K_n/\Aff(\mathbb C)
\).

Recall that the space \(\mathcal P_n/\Aff(\mathbb C)\) is identified with the hypersurface
\(
y_1\cdots y_n=1
\)
in the torus \((\C^\times)^n\). It carries the Poisson structure
\eqref{eq:logcanonical}. 
For even
\(n=2m\), this Poisson structure  has a Casimir function
\[
C:=y_1y_3\cdots y_{2m-1}.
\]

\begin{proposition}\label{prop:Im-C}
Let \(n=2m\). Then the integral \(I_m\) defined in
Section~\ref{sec:fi} satisfies
\begin{equation}\label{eq:ICrel}
    I_m=C+(-1)^mC^{-1}. 
\end{equation}
In particular, \(I_m\) is a Casimir of the Poisson structure.
\end{proposition}

\begin{proof}
The integral $I_m$ is the coefficient of \(t^m\) in
\[
\tr\bigl(Y_1(t)\cdots Y_{2m}(t)\bigr)
\]
where $Y_i(t)$ is given by \eqref{eq:ymatrix}. So,
\[
I_m
=
\varepsilon_{2m}y_1\varepsilon_{2}y_3\cdots\varepsilon_{2m-2}y_{2m-1}
+
\varepsilon_1y_2 \varepsilon_{3}y_4\cdots\varepsilon_{2m-1}y_{2m}. 
\]
Since
\(
y_1\cdots y_{2m}=1,
\)
this becomes
\[
I_m
=
(\varepsilon_2\cdots\varepsilon_{2m})\,C + (\varepsilon_1\cdots\varepsilon_{2m-1})\,C^{-1}.
\]
For staircase initial data,
\[
\varepsilon_1\cdots\varepsilon_{2m-1}=(-1)^m,
\qquad
\varepsilon_2\cdots\varepsilon_{2m}=1,
\]
which proves the identity.
\end{proof}

\begin{proposition}
The multikite space \(\mathcal K_{2m}/\Aff(\mathbb C)\) is contained in the symplectic
leaf
\begin{equation}\label{eq:leaf}
C=\pm (\sqrt{-1})^m.
\end{equation}
\end{proposition}

\begin{remark}
Since \(C\) is the only Casimir of the Poisson structure on
\(\mathcal P_{2m}\), the restriction of the Poisson
structure to \(C=\pm (\sqrt{-1})^m\) is nondegenerate. Hence
\(C=\pm (\sqrt{-1})^m\) may be viewed as a symplectic leaf. Here by a
symplectic leaf we mean a Poisson submanifold on which the Poisson
structure is nondegenerate, not necessarily connected.
\end{remark}

\begin{remark}\label{rem:leaf2}
In view of \eqref{eq:ICrel}, equation \eqref{eq:leaf} implies
\begin{equation}\label{eq:leaf2}
I_m=\pm 2(\sqrt{-1})^m. 
\end{equation}
Conversely, for \(I_m\) of this form, equation \eqref{eq:ICrel} has a
double root
\(
C=\pm(\sqrt{-1})^m.
\)
Thus the leaf \eqref{eq:leaf} can equivalently be described by
\eqref{eq:leaf2}.
\end{remark}

\begin{proof}[Proof of the Proposition]
For a right multikite, the defining conditions imply that the even
variables \(y_i\) lie on the imaginary axis, while the odd variables
\(y_i\) lie on the unit circle. From the latter we see that
\[
C=y_1y_3\cdots y_{2m-1}
\]
lies on the unit circle. On the other hand, since
\(
y_2,y_4,\dots,y_{2m}\in \sqrt{-1}\,\mathbb R
\)
and
\(
y_1y_2\cdots y_{2m}=1,
\)
we have
\[
C^{-1}=y_2y_4\cdots y_{2m}\in (\sqrt{-1})^m\mathbb R.
\]
Hence
\(
C\in (\sqrt{-1})^m\mathbb R.
\)
Thus
\[
C\in S^1\cap (\sqrt{-1})^m\mathbb R
=
\{\pm (\sqrt{-1})^m\}.\qedhere
\]
\end{proof}

Our next goal is to show that the multikite space is the real part of
this symplectic leaf, for a suitable real structure.

\begin{definition}
A \emph{real structure} on an affine variety \(X\) over \(\C\) is an
anti-regular involution
\(
\tau\colon X\to X,
\)
see \cite[Definition~2.1.10]{mangolte2020real}. The corresponding
\emph{real part} of \(X\) is the fixed-point set
\[
X^\tau=\{x\in X:\tau(x)=x\},
\]
see \cite[Definition~2.1.22]{mangolte2020real}.
\end{definition}

Any complex affine variety \(X\) with a real structure \(\tau\) admits
a closed embedding into \(\C^N\) that intertwines \(\tau\) with
coordinate-wise complex conjugation on \(\C^N\); see
\cite[Theorem~2.1.30]{mangolte2020real}. Replacing \(X\) by its image
under this embedding, we may therefore assume that \(X\) is cut out by
real polynomials and
\[
X^\tau=X\cap\mathbb R^N.
\]
In particular, \(X^\tau\) is a real affine variety. Its coordinate ring
is naturally identified with the ring of functions on \(X\) fixed by the
\(\C\)-antilinear involution
\begin{equation}\label{eq:ii}
f\longmapsto\overline{\tau^*f}.
\end{equation}
We refer to these functions as \emph{real functions}. Under the
embedding above, they are precisely the regular functions on \(X\) with
real coefficients.


\begin{definition}
A Poisson structure on a complex affine variety \(X\) with a real
structure \(\tau\) is called \emph{real} if the induced involution
\eqref{eq:ii} on the coordinate ring preserves the Poisson bracket, that
is,
\[
\{\overline{\tau^*f},\overline{\tau^*g}\}
=
\overline{\tau^*\{f,g\}}.
\]
In particular, if \(X\) carries a real Poisson structure, then the
Poisson bracket of real functions is again real. Hence, via the
identification of the coordinate ring of \(X^\tau\) with the ring of
real functions on \(X\), it gives a Poisson structure on \(X^\tau\),
which we call the \emph{induced Poisson structure}.
\end{definition}


In our setting, the real structure on the leaf
\[
C=\pm (\sqrt{-1})^m
\]
is defined by $\tau \colon (y_1,\dots, y_{2m}) \mapsto (y_1',\dots, y_{2m}') $ where
\[
y_i'=
\begin{cases}
\overline{y}_i^{-1}, & i \text{ odd},\\
-\overline{y}_i, & i \text{ even}.
\end{cases}
\]

\begin{proposition}\label{prop:realpart}
The map \(\tau\) is a well-defined real structure on the symplectic leaf
\(C=\pm(\sqrt{-1})^m\). Its real part is \(\mathcal K_{2m}/\Aff(\mathbb C)\).
\end{proposition}

\begin{proof}
We first check that \(\tau\) preserves the leaf. Recall that the leaf is
defined inside \((\C^\times)^{2m}\) by
\[
y_1\cdots y_{2m}=1,
\qquad
C:=y_1y_3\cdots y_{2m-1}
=
\pm(\sqrt{-1})^m.
\]
Under \(\tau\),
\[
\tau^*C
=
\overline y_1^{-1}\overline y_3^{-1}\cdots\overline y_{2m-1}^{-1}
=
\overline C^{-1}
=
C.
\]
Here the last equality follows from the fact that
\(C=\pm(\sqrt{-1})^m\), and hence \(|C|=1\). 
Thus the condition \(C=\pm(\sqrt{-1})^m\) is preserved. Similarly,
\[
\begin{gathered}
\tau^*(y_1\cdots y_{2m})
=
\frac{(-\overline y_2)(-\overline y_4)\cdots(-\overline y_{2m})}
{\overline y_1\overline y_3\cdots\overline y_{2m-1}} 
=
(-1)^m
\frac{\overline{y_2 y_4\cdots y_{2m}}}
{\overline{y_1y_3\cdots y_{2m-1}}} \\
=
(-1)^m
\frac{\overline{C^{-1}}}{\overline C}
=
\frac{(-1)^m}{\overline C^{\,2}}
=
\frac{(-1)^m}{\bigl(\pm(-\sqrt{-1})^m\bigr)^2}
=
1.
\end{gathered}
\]
Thus \(\tau\) preserves the defining equations of the leaf. It is clearly anti-regular
and satisfies \(\tau^2=\mathrm{id}\), so it is a real structure. Its fixed point set is given by
\[
y_i=
\begin{cases}
\overline{y}_i^{-1}, & i \text{ odd},\\
-\overline{y_i}, & i \text{ even},
\end{cases}
\]
or equivalently,
\[
|y_i|=1,\ i\text{ odd},
\qquad
y_i\in\sqrt{-1}\,\mathbb R,\ i\text{ even}.
\]
These are precisely the defining conditions for \(\mathcal K_{2m}\).
Hence the real part of the symplectic leaf is \(\mathcal K_{2m}/\Aff(\mathbb C)\).
\end{proof}

\begin{proposition}\label{prop:realPoisson}
The Poisson structure
\(
\sqrt{-1}\{\cdot,\cdot\},
\)
where \(\{\cdot,\cdot\}\) is the Poisson bracket~\eqref{eq:logcanonical}, is real with respect to the real structure
\(\tau\) on the leaf \(C=\pm (\sqrt{-1})^m\).
\end{proposition}

\begin{proof}
Let
\(
\hat \tau(f)=\overline{\tau^*f}
\)
be the induced \(\C\)-antilinear involution on regular functions. It is
enough to check the statement on the coordinate functions. Recall that
\[
\{y_i,y_{i+1}\}=y_iy_{i+1},
\]
while all other brackets vanish. By the definition of \(\tau\),
\[
\hat \tau(y_i)=
\begin{cases}
y_i^{-1}, & i \text{ odd},\\
-y_i, & i \text{ even}.
\end{cases}
\]
Assume first that \(i\) is odd. Then
\[
\begin{gathered}
\{\hat \tau(y_i),\hat \tau(y_{i+1})\}
=\{y_i^{-1},-y_{i+1}\}
=y_i^{-2}\{y_i,y_{i+1}\} \\ 
=y_i^{-1}y_{i+1}
=-\hat \tau(y_iy_{i+1})
=-\hat \tau(\{y_i,y_{i+1}\}).
\end{gathered}
\]
The case of even \(i\) is analogous. For non-adjacent indices, both
sides vanish. Thus the identity
\[
\{\hat \tau(f),\hat \tau(g)\}=-\hat \tau(\{f,g\})
\]
holds for all coordinate functions \(f=y_i\), \(g=y_j\), and by the
Leibniz rule it extends to all regular functions. Therefore the original
Poisson bracket is anti-real with respect to \(\tau\). Since
\(\hat \tau(\sqrt{-1})=-\sqrt{-1}\), the bracket
\(\sqrt{-1}\{\cdot,\cdot\}\) is real.
\end{proof}
As a consequence of the proposition, the Poisson bracket
\(
\sqrt{-1}\{\cdot,\cdot\}
\)
on the symplectic leaf
\(
C=\pm (\sqrt{-1})^m
\)
induces a (nondegenerate) real Poisson bracket on its real part
\(\mathcal K_{2m}/\Aff(\mathbb C)\).

\begin{remark}
Writing
\[
y_i=
\begin{cases}
e^{\sqrt{-1}\phi_i}, & i\text{ odd},\\
\sqrt{-1}\,\ell_i, & i\text{ even},
\end{cases}
\]
we obtain a real parametrization of \(\mathcal K_{2m}/\Aff(\mathbb C)\). The parameters $\phi_i \in \mathbb R / 2\pi\Z$, $\ell_i \in \mathbb R^\times$ are subject to relations 
\[
\phi_1+\cdots+\phi_{2m-1}
=
\delta\frac{\pi m}{2},
\qquad
\ell_2\cdots\ell_{2m}
=
\delta(-1)^m,
\]
where \(\delta=\pm1\) is the sign in~\eqref{eq:leaf}. In these coordinates,
\[
\{\ell_{2k},\phi_{2k-1}\}
=
\{\ell_{2k},\phi_{2k+1}\}
=
\ell_{2k},
\]
and all other brackets vanish.
\end{remark}


\subsection{The action of the real structure on the integrals}

We now turn to the behaviour of the integrals
$I_j$
under the involution \(\tau\). In this section we show that,
although the functions \(I_j\) are not generally fixed by \(\tau\),
suitable combinations of them are, and use these combinations to obtain
first integrals on the multikite space \(\mathcal K_{2m}/\Aff(\mathbb C)\).

Let
\[
\mathbf Y(t)=Y_1(t)\cdots Y_{2m}(t),
\]
where the matrices \(Y_i(t)\) are given by \eqref{eq:ymatrix}, with
\(
\varepsilon_i=(-1)^i
\)
corresponding to staircase initial data.
Recall that
\[
\operatorname{tr}(\mathbf Y(t))
=
I_0+I_1t+\cdots+I_mt^m,
\]
where \(I_0=2\), and \(I_1,\dots,I_m\) are commuting independent first
integrals. 
\begin{proposition}\label{prop:intrestr}
On the symplectic leaf
\(
C=\pm(\sqrt{-1})^m,
\)
the antilinear involution \(\hat \tau\) on regular functions induced by the
real structure \(\tau\) satisfies
\begin{equation}\label{eq:sigmaformula}
\hat \tau(I_k)=C^{-1}\,I_{m-k},
\qquad
k=0,\dots,m. 
\end{equation}
\end{proposition}

\begin{remark}
Since \(I_0=2\) and \(\hat \tau\) acts trivially on real constants, for $k=0$ this gives
\[
I_m = C\hat \tau(I_0) = 2C, 
\]
recovering the already familiar identity~\eqref{eq:leaf2}.
\end{remark}

\begin{proof}[Proof of the proposition]
Write
\[
\mathbf Y(t)=\widetilde Y_1(t)\cdots \widetilde Y_m(t),
\qquad
\widetilde Y_k(t):=Y_{2k-1}(t)Y_{2k}(t).
\]
A direct computation gives
\[
\widetilde Y_k(t)=
\begin{pmatrix}
1+y_{2k-1}t
&
y_{2k}t(y_{2k-1}-1)
\\[2mm]
y_{2k-1}+1
&
y_{2k}(y_{2k-1}-t)
\end{pmatrix}.
\]
Using
\(
\hat \tau(y_{2k-1})=y_{2k-1}^{-1}
\) and \(
\hat \tau(y_{2k})=-y_{2k},
\)
we obtain
\[
t\,y_{2k-1}\,\hat \tau\!\left(\widetilde Y_k(t^{-1})\right)
=
\,D\widetilde Y_k(t)D^{-1},
\qquad
D=\operatorname{diag}(1,t).
\]
Multiplying these identities for \(k=1,\dots,m\), we get
\[
t^m C\,\hat \tau\!\left(\mathbf Y(t^{-1})\right)
=
D\mathbf Y(t)D^{-1}, 
\]
where
\(
C:=y_1y_3\cdots y_{2m-1}.
\)
Taking traces gives
\[
t^m C\,\hat \tau\!\left(\tr\mathbf Y(t^{-1})\right)
=
\tr\mathbf Y(t). 
\]
Substituting
\[
\tr\mathbf Y(t)
=
I_0+\cdots+I_mt^m
\]
and comparing coefficients, we obtain
\[
C\,\hat \tau(I_k)= I_{m-k}. \qedhere
\]
\end{proof}

Fix one of the two irreducible components of the symplectic leaf
$C=\pm(\sqrt{-1})^m$, and let
\[
V:=\mathrm{span}_{\C}\{I_1,\dots,I_{m-1}\},
\]
where the functions \(I_j\) are regarded as functions on that component.
By Proposition~\ref{prop:intrestr}, the space \(V\) is invariant under
the antilinear involution \(\hat\tau\). Therefore, \(V\) admits a basis
\(
J_1,\dots,J_{m-1}
\)
consisting of real functions, i.e.
\[
\hat\tau(J_k)=J_k,
\qquad
k=1,\dots,m-1.
\]
We call the functions \(J_1,\dots,J_{m-1}\) the \emph{real integrals}.


\begin{proposition}\label{prop:Jkfuncs}
The restrictions of the real integrals $J_k$ to the multikite
space \(\mathcal K_{2m}/\Aff(\mathbb C)\) pairwise Poisson commute. 
\end{proposition}

\begin{proof}

Since the functions \(I_1,\dots,I_{m-1}\) pairwise Poisson commute on the
complex symplectic leaf, any linear combinations of them also pairwise
Poisson commute. Hence the functions \(J_k\) pairwise commute. Restricting
to the real part \(\mathcal K_{2m} /\Aff(\mathbb C)\) gives commuting functions on the
multikite space.
\end{proof}
\begin{remark}
One possible choice of the real integrals is
\[
J_k=\alpha I_k+\overline{\alpha}\,C^{-1}I_{m-k},
\]
where \(\alpha\in\C^\times\) satisfies
\[
\left(\frac{\overline{\alpha}}{\alpha}\right)^2\neq(-1)^m.
\]
Indeed, \(J_k\) is a real function, and the resulting linear
transformation
\[
(I_1,\dots,I_{m-1})
\longmapsto
(J_1,\dots,J_{m-1})
\]
is invertible.
\end{remark}

\subsection{Independence of the restricted integrals}
In the previous section, we constructed \(m-1\) pairwise
Poisson-commuting functions
\(
J_1,\dots,J_{m-1}
\)
on the symplectic variety
\(
\mathcal K_{2m}/\Aff(\mathbb C),
\)
which has dimension \(2m-2\). Thus, to establish complete
integrability, it remains to prove that these functions are
algebraically independent.

Since the multikite space
\(
\mathcal K_{2m}/\Aff(\mathbb C)
\)
is the real locus of the symplectic leaf~\eqref{eq:leaf},
algebraic independence of the functions \(J_k\) on
the mutlikite space 
is equivalent to their independence on the leaf. Moreover, since the
linear transformation
\[
(I_1,\dots,I_{m-1})
\longmapsto
(J_1,\dots,J_{m-1})
\]
is invertible, it is enough to prove that the restrictions of
\(
I_1,\dots,I_{m-1}
\)
to the leaf~\eqref{eq:leaf} are algebraically independent.

Although the functions \(I_1,\dots,I_{m-1}\) are already known to be
independent on
\(
\mathcal P_{2m}/\Aff(\mathbb C),
\)
this does not immediately imply their independence on the symplectic
leaf. The remainder of this section is devoted to proving this fact,
thereby completing the proof of Theorem~\ref{thm2}.

It is enough to prove that the restrictions of the integrals
\(
I_1,\dots,I_{m-1}
\)
to each irreducible component of the symplectic leaf
~\eqref{eq:leaf}
are algebraically independent. Accordingly, we fix
\(
\nu\in\{\pm1\}
\)
and work on the component
\(
C=\nu(\sqrt{-1})^m.
\)
This component is equivalently described by the
equation
\[
I_m=2\nu(\sqrt{-1})^m, 
\]
see Remark \ref{rem:leaf2}. 
Following Section~\ref{sec:iqp}, we will prove independence on the
quotient by translations
\(
\mathcal P_{2m}/\C.
\)
In terms of the coordinates
\(
z_i=p_i-p_{i-1}
\)
on that space, the Casimir $C$ is given by
\[
C
=
\frac{z_1z_3\cdots z_{2m-1}}
     {z_2z_4\cdots z_{2m}}.
\]
Therefore, it suffices to prove the following.
\begin{proposition}\label{prop:I-dominant-leaf}
Let
\[
H:=
\left\{
(z_1,\dots,z_n)\in(\C^\times)^n
\;\middle|\;
\frac{z_1z_3\cdots z_{2m-1}}
     {z_2z_4\cdots z_{2m}}
=
\nu(\sqrt{-1})^m
\right\}.
\]
Then the map
\[
H\longrightarrow \C^{m-1},
\qquad
(z_1,\dots,z_n)
\longmapsto
(I_1,\dots,I_{m-1}),
\]
where the functions \(I_j\) are defined as in
Proposition~\ref{prop:I-dominant}, with
\(
\varepsilon_i=(-1)^i,
\)
is dominant.
\end{proposition}
Keeping the notation of Section~\ref{sec:iqp}, let
\[
\widetilde{\mathcal Z}_n
:=
\left\{
\mathbf Z(t)\in\mathcal Z_n
\;\middle|\;
I_m=2\nu(\sqrt{-1})^m
\right\}, 
\]
where \(I_m\) denotes the leading coefficient of
\(\tr\mathbf Z(t)\), and let
\[
\widetilde{\mathcal Z}_n^\circ
:=
\widetilde{\mathcal Z}_n\cap\mathcal Z_n^\circ.
\]
The factorizable locus in \(\widetilde{\mathcal Z}_n^\circ\) is the set
of matrices \(\mathbf Z(t)\in\widetilde{\mathcal Z}_n^\circ\) admitting
a factorization
\[
\mathbf Z(t)=Z_1(t)\cdots Z_n(t),
\]
where \(Z_i(t)\) is given by \eqref{eq:zmatrix} with
\(
\varepsilon_i=(-1)^i,
\)
for some
\[
z = (z_1,\dots,z_n)\in(\C^\times)^n.
\]
By the definition of \(\widetilde{\mathcal Z}_n\), such a point \(z\)
necessarily belongs to \(H\).

\begin{lemma}\label{lem:nonemptiness-leaf}
The factorizable locus in
\(\widetilde{\mathcal Z}_n^\circ\)
is nonempty.
\end{lemma}

\begin{proof}
We need to show that there exists a point
\(
(z_1,\dots,z_n)\in H
\)
such that
\(
\mathbf Z(\pm1)\neq0.
\)
For such a point, we then have~\(\mathbf Z(t) \in \widetilde{\mathcal Z}_n^\circ\).

Since \(\varepsilon_1=-1\), the matrix \(Z_1(1)\) is invertible. Hence
\[
\mathbf Z(1)\neq0
\iff
Z_2(1)\cdots Z_n(1)\neq0.
\]
In other words, the subset of \(H\) on which
\(\mathbf Z(1)\neq0\) is the preimage of
\begin{equation}\label{eq:prod2n}
\{(z_2,\dots,z_n)\in(\C^\times)^{n-1}
\mid
Z_2(1)\cdots Z_n(1)\neq0\} \subset (\C^\times)^{n-1}
\end{equation}
under the isomorphism
\[
H\to(\C^\times)^{n-1},
\qquad
(z_1,\dots,z_n)\longmapsto(z_2,\dots,z_n).
\]
The set \eqref{eq:prod2n} is Zariski open by definition and nonempty by
Lemma~\ref{lem:nonemptiness}. Hence the subset of \(H\) on which
\(\mathbf Z(1)\neq0\) is also nonempty and Zariski open.

Similarly, since \(\varepsilon_n=1\) and hence the matrix \(Z_n(-1)\) is
invertible, the subset of \(H\) on which
\(\mathbf Z(-1)\neq0\) is also nonempty and Zariski open. Thus, the
subset of \(H\) on which
\(\mathbf Z(\pm1)\neq0\) is the intersection of two nonempty Zariski
open subsets. Since \(H\) is irreducible, this intersection must be
nonempty.
\end{proof}



We are now in a position to prove
Proposition~\ref{prop:I-dominant-leaf}.

\begin{proof}[Proof of Proposition~\ref{prop:I-dominant-leaf}]
The map
\[
H
\longrightarrow
\C^{m-1},
\qquad
(z_1,\dots,z_n)
\longmapsto
(I_1,\dots,I_{m-1})
\]
is the composition of the following two maps:
\begin{align}
\mathrm{product}\colon\quad
H &\longrightarrow \widetilde{\mathcal Z}_n,
&
(z_1,\dots,z_n)
&\longmapsto
Z_1(t)\cdots Z_n(t), \label{eq:product-map}\\
\mathrm{trace}\colon\quad
\widetilde{\mathcal Z}_n &\longrightarrow \C^{m-1},
&
\mathbf Z(t)
&\longmapsto
(I_1,\dots,I_{m-1}). \label{eq:trace-map}
\end{align}

By Lemma~\ref{lem:openness}, the factorizable locus in
\(\mathcal Z_n^\circ\) is Zariski open. Hence its intersection with
\(\widetilde{\mathcal Z}_n^\circ\), which is precisely the factorizable
locus in \(\widetilde{\mathcal Z}_n^\circ\), is Zariski open in
\(\widetilde{\mathcal Z}_n^\circ\). Furthermore, by
Lemma~\ref{lem:nonemptiness-leaf}, this locus is nonempty. Thus the
image of the product map~\eqref{eq:product-map} contains a nonempty
Zariski open subset of \(\widetilde{\mathcal Z}_n^\circ\), and hence of
\(\widetilde{\mathcal Z}_n\). Therefore the product map is dominant.

The trace map~\eqref{eq:trace-map} is the restriction of the
map
\[
\mathcal Z_n
\longrightarrow
\C^m,
\qquad
\mathbf Z(t)\longmapsto(I_1,\dots,I_m),
\]
which is surjective by Lemma~\ref{lem:trace-surjectivity}, to
\[
\widetilde{\mathcal Z}_n
=
\{\mathbf Z(t)\in\mathcal Z_n
\mid
I_m=2\nu(\sqrt{-1})^m\}.
\]
Hence the trace map~\eqref{eq:trace-map} is also surjective, and
therefore dominant.

Since the composition of dominant maps is dominant, the map
\[
H
\longrightarrow
\C^{m-1},
\qquad
(z_1,\dots,z_n)
\longmapsto
(I_1,\dots,I_{m-1})
\]
is dominant.
\end{proof}

\subsection{Complete integrability of the Cauchy problem (Proof of Theorem~\ref{thm2})}

\begin{proof}[Proof of Theorem~\ref{thm2}]
By Proposition~\ref{prop:realpart}, the multikite space
\(\mathcal K_{2m}/\Aff(\mathbb C)\) is the real part of the symplectic leaf
\[
\{\,C=\pm(\sqrt{-1})^m\,\}
\subset
\mathcal P_{2m} /\Aff(\mathbb C).
\]
By Proposition~\ref{prop:realPoisson}, the Poisson structure
\(
\sqrt{-1}\{\cdot,\cdot\}
\)
is real with respect to the corresponding real structure. Since the
underlying complex Poisson structure is nondegenerate on the symplectic
leaf, its restriction to the real part is also nondegenerate. Thus
\(\mathcal K_{2m}/\Aff(\mathbb C)\) is a real symplectic variety.

By Proposition~\ref{prop:Jkfuncs}, the  restrictions of the real integrals $J_1, \dots, J_{m-1}$ to \(\mathcal K_{2m} /\Aff(\mathbb C)\) pairwise
Poisson commute. By
Proposition~\ref{prop:I-dominant-leaf}, they are algebraically
independent. Since
\[
m-1=\frac12\dim\mathcal K_{2m}/\Aff(\mathbb C),
\]
they form a complete set of commuting first integrals.

Finally, the solution map \(\mathcal S\) preserves the functions
\(I_1,\dots,I_{m-1}\), and hence also the functions
\(J_1,\dots,J_{m-1}\). Therefore the restricted map
\begin{equation}
\mathcal S|_{\mathcal K_n/\Aff(\C)}\colon
\mathcal K_n/\Aff(\C)
\dashrightarrow
\mathcal K_n/\Aff(\C).
\end{equation}
is a completely integrable symplectic map.
\end{proof}

\section{Open problems}
A longstanding problem in the theory of discrete integrable systems is to
establish a precise relationship between integrability in the sense of
multidimensional consistency and Arnold--Liouville integrability. In the
present paper we resolve this problem for the cross-ratio equation, one of
the most fundamental examples of a multidimensionally consistent
quad-equation. Nevertheless, the cross-ratio equation represents only one
member of the Adler--Bobenko--Suris classification~\cite{adler2003classification}.

\begin{problem}
Establish Arnold--Liouville integrability of initial value problems for the
remaining multidimensionally consistent quad-equations in the Adler--Bobenko--Suris
classification.
\end{problem}

Of particular importance is the elliptic equation $Q4$, which occupies the
top level of the Adler--Bobenko--Suris hierarchy, with all other equations arising from it
via suitable degeneration procedures.

A particular instance of the previous problem is the
following.

\begin{problem}
Extend the results of this paper to the most general multidimensionally
consistent cross-ratio equation
\[
[f_{i,j},f_{i+1,j},f_{i+1,j+1},f_{i,j+1}]
=\frac{\alpha_i}{\beta_j},
\]
where $\alpha_i$ and $\beta_j$ are prescribed sequences in
$\mathbb C^\times$.
\end{problem}

The above equation in particular governs square grid circle patterns with
prescribed intersection angles, generalizing the orthogonal circle
patterns studied in the present paper; see
~\cite[Chapter~8]{BobenkoSuris2008}.

\begin{problem}
Extend the results of this paper from orthogonal square grid circle
patterns to square grid circle patterns with prescribed intersection
angles.
\end{problem}

The next problem concerns periodic versions of the systems considered in
this paper. Instead of the quasi-periodicity condition
\[
f_{x+\per}=f_x+\Delta,
\]
where $\per$ is a period vector and $\Delta\in\mathbb C$ is the monodromy, one
may consider the purely periodic case
\[
f_{x+\per}=f_x.
\]
In this setting, the corresponding spaces of initial data for discrete
conformal maps and orthogonal square grid circle patterns carry a natural
action of the full M\"obius group $\PGL_2(\mathbb C)$, and the
corresponding solution maps descend to the quotients by this action.
However, our approach does not extend directly to the periodic setting:
the embedding of the space of initial data for orthogonal square grid
circle patterns into the space of initial data for discrete conformal maps
is defined using circle centers and is therefore not M\"obius equivariant.

\begin{problem}
Study the initial value problems for periodic discrete conformal maps and
periodic orthogonal square grid circle patterns, and establish
Arnold--Liouville integrability of the corresponding solution maps on the
quotients of the spaces of initial data by the M\"obius group.
\end{problem}

Finally, it would also be interesting to compare the symplectic structure
on the space of initial data for orthogonal square grid circle patterns
constructed in this paper with previously known symplectic structures on
spaces of circle patterns.

\begin{problem}
Determine the relationship between the symplectic structure on the space
of initial data for orthogonal square grid circle patterns constructed in
this paper and Lam's symplectic structure on spaces of circle patterns
arising from Teichm\"uller theory~\cite{lam2024symplectic}.
\end{problem}

\bibliographystyle{plain}
\bibliography{main.bib}

@article{nijhoff1995discrete,
  title={The discrete {K}orteweg-de {V}ries equation},
  author={Nijhoff, F. and Capel, H.},
  journal={Acta Appl. Math.},
  volume={39},
  pages={133--158},
  year={1995},
  publisher={Springer}
}

@article{bobenko2005linear,
  title={Linear and nonlinear theories of discrete analytic functions. {I}ntegrable structure and isomonodromic {G}reen’s function},
  author={Bobenko, A. and Mercat, C. and Suris, Y.},
  year={2005},
  journal = {J. Reine Angew. Math.},
  volume = {583},
  pages = {117--161}
}

@article{Izmestiev2023,
  author = {Izmestiev, I.},
  title = {Deformation of quadrilaterals and addition on elliptic curves},
  journal = {Mosc. Math. J.},
  volume = {23},
  pages = {205--242},
  year = {2023}
}

@article{bobenko1996discrete,
  title={Discrete isothermic surfaces.},
  author={Bobenko, A. and Pinkall, U.},
  year={1996},
 journal = {J. Reine Angew. Math.},
 volume = {475}, 
 pages = {187--208}
}

@article{hetrich2001periodic,
  title={Periodic discrete conformal maps},
  author={Hetrich-Jeromin, U. and McIntosh, I. and  Norman, P. and Pedit, F.},
  year={2001},
  volume = {534},
journal = {J. Reine Angew. Math.},
  pages = {129--153}
}

@incollection{bobenko1999discrete,
  title={Discrete conformal maps and surfaces},
  author={Bobenko, A.},
  booktitle={Symmetries and Integrability of Difference equations (Canterbury, 1996)},
  editor = {Clarkson, P.A.  and Nijhoff,  F.W. },
  series = {London Mathematical Society Lecture Note Series},
  pages={97--108},
  year={1999},
  publisher={Cambridge University Press},
  volume = {255}
}

@article{schramm1997circle,
  title={Circle patterns with the combinatorics of the square grid},
  author={Schramm, O.},
  year={1997},
  journal = {Duke Math. J.},
  volume = {86},
  number = {2},
  pages = {347-389}
}

@article{bobenko2002integrable,
  title={Integrable systems on quad-graphs},
  author={Bobenko, A. and Suris, Y.},
  journal={Int. Mat. Res. Not.},
  volume={2002},
  number={11},
  pages={573--611},
  year={2002},
  publisher={Hindawi Publishing Corporation}
}

@article{arnold2022cross,
  title={Cross-ratio dynamics on ideal polygons},
  author={Arnold, M. and Fuchs, D. and Izmestiev, I. and Tabachnikov, S.},
  journal={Int. Math. Res. Not.},
  volume={2022},
  number={9},
  pages={6770--6853},
  year={2022},
  publisher={Oxford University Press}
}

@article{izosimov2022pentagram,
  author  = {Izosimov, A.},
  title   = {Pentagram maps and refactorization in Poisson-Lie groups},
  journal = {Adv. Math.},
  volume  = {404},
  year    = {2022},
  pages   = {108476},
  doi     = {10.1016/j.aim.2022.108476},
  eprint  = {arXiv:1803.00726}
}

@article{adler2003classification,
  author  = {V. Adler and A. Bobenko and Y. Suris},
  title   = {Classification of Integrable Equations on Quad-Graphs. The Consistency Approach},
  journal = {Comm. Math. Phys.},
  volume  = {233},
  number  = {3},
  pages    = {513--543},
  year     = {2003}
}

@misc{lam2024symplectic,
  author        = {Lam, W.~Y.},
  title         = {Space of circle patterns on tori and its symplectic form},
  year          = {2024},
  eprint        = {2406.06733},
  archivePrefix = {arXiv},
  primaryClass  = {math.GT},
  note          = {Preprint, arXiv:2406.06733}
}

@book{BobenkoSuris2008,
  author    = {A. Bobenko and Y. Suris},
  title     = {Discrete Differential Geometry: Integrable Structure},
  series    = {Graduate Studies in Mathematics},
  volume     = {98},
  publisher = {American Mathematical Society},
  address   = {Providence, RI},
  year      = {2008},
  isbn       = {978-0-8218-4700-8}
}

@misc{DragovicRadnovic2025,
  author        = {Dragovi\'c, V. and Radnovi\'c, M.},
  title         = {Finite groups of random walks in the quarter plane and periodic 4-bar links},
  year          = {2025},
  eprint        = {2512.21976},
  archivePrefix = {arXiv},
  primaryClass  = {math.DS},
  note          = {Preprint, arXiv:2512.21976}
}

@article{adler2004cauchy,
  title={Cauchy problem for integrable discrete equations on quad-graphs},
  author={Adler, V. and Veselov, A.},
  journal={Acta Appl. Math.},
  volume={84},
  pages={237--262},
  year={2004},
  publisher={Springer}
}

@book{mangolte2020real,
  title={Real algebraic varieties},
  author={Mangolte, F.},
  year={2020},
  publisher={Springer}
}

@incollection {Coxeterquiver_Tomas,
    AUTHOR = {Thomas, H.},
     TITLE = {Coxeter groups and quiver representations},
 BOOKTITLE = {Surveys in representation theory of algebras},
    SERIES = {Contemp. Math.},
    VOLUME = {716},
     PAGES = {173--186},
 PUBLISHER = {Amer. Math. Soc},
      YEAR = {2018},
      ISBN = {978-1-4704-3679-7},
   MRCLASS = {16G20 (20F55)},
  MRNUMBER = {3852402},
MRREVIEWER = {Mar\'{\i}a\ Julia\ Redondo},
       DOI = {10.1090/conm/716/14429},
       URL = {https://doi.org/10.1090/conm/716/14429},
}
\end{document}